\documentclass[10pt, reqno, a4paper]{amsart}
\usepackage{amssymb,latexsym,amsmath,amsfonts,amsthm}
\usepackage{latexsym}
\usepackage[numbers,sort&compress]{natbib}
\usepackage[pagewise]{lineno}
\usepackage{color}
\usepackage[mathscr]{eucal}
\usepackage{comment}
\usepackage{dsfont}
\newcommand{\mc}[1]{\mathcal{#1}}
\addtocontents{toc}{\protect\setcounter{tocdepth}{1}}

\usepackage[linkcolor=red, citecolor=red]{hyperref}
\hypersetup{colorlinks=true, urlcolor=blue}
\usepackage[nameinlink]{cleveref}

\usepackage[top=2.7cm,bottom=2.7cm,left=2.7cm,right=2.7cm]{geometry}

\numberwithin{equation}{section}

\theoremstyle{definition}
\newtheorem{definition}{Definition}[section]

\theoremstyle{remark}
\newtheorem{remark}[definition]{Remark}
\theoremstyle{plain}
\newtheorem{theorem}[definition]{Theorem}
\newtheorem{lemma}[definition]{Lemma}
\newtheorem{proposition}[definition]{Proposition}

\newcommand{\norm}[1]{\left\|#1\right\|}

\newcommand{\wphi}{\widehat{\zeta}}

\newcommand{\vphi}{\varphi}

\newcommand \F{\mathcal S}

\newcommand\Y{\mathcal{Y}}
\newcommand\V{\mathcal{V}}

\newcommand\C{\mathcal{C}}

\newcommand{\iintQ}{\iint_{Q_T}}

\begin{document}
	\title[Controllability of the Cahn-Hilliard Burgers' equation]{Local controllability of the Cahn-Hilliard-Burgers' equation around certain steady states}
	\author [Manika Bag, Sheetal Dharmatti, Subrata Majumdar, Debanjana Mitra]{Manika Bag\textsuperscript{1}, Sheetal Dharmatti\textsuperscript{2}, Subrata Majumdar\textsuperscript{3}, Debanjana Mitra\textsuperscript{4 *} }


	\subjclass[2020]{35M33, 35K61, 93B05, 93B07, 93C20 }
\keywords{ Controllability, Carleman estimates,  source term method, fixed point argument}
\date{\today}
\thanks
{
\textsuperscript{1,2} School of Mathematics,
			Indian Institute of Science Education and Research, Thiruvananthapuram (IISER-TVM),
			Maruthamala PO, Vithura, Thiruvananthapuram, Kerala, 695551, INDIA.  
			\textit{e-mail:} \texttt{manikabag19@iisertvm.ac.in,}		\texttt{sheetal@iisertvm.ac.in} \\
\textsuperscript{3}Instituto de Matemáticas,Universidad Nacional Autónoma de México,
Circuito Exterior, Ciudad Universitaria,
04510 Coyoacán, Ciudad de México, México.
\textit{e-mail:} \texttt{subrata.majumdar@im.unam.mx}\\				
\textsuperscript{4}Department of Mathematics, Indian Institute of Technology Bombay, Powai, Mumbai,  400076, INDIA
\textit{e-mail:}	\texttt{deban@math.iitb.ac.in, debanjana.math@gmail.com \\} 
\textsuperscript{*}  Corresponding author\\
{\textbf{Acknowledgments:}} Manika Bag would like to thank the  Indian Institute of Science Education and Research, Thiruvananthapuram, for providing financial support and stimulating  research environment.  Sheetal Dharmatti would like to thank  Science $\&$ Engineering Research Board (SERB), India, Core Research grant, CRG/2021/008278 for financial support. Subrata Majumdar received financial support from the institute post-doctoral fellowship of IIT Bombay, India and the post-doctoral scholarship of the UNAM, México. 
}

\begin{abstract}
In this article we study the local controllability of the one-dimensional Cahn-Hilliard-Navier-Stokes equation, that is Cahn-Hilliard-Burgers' equation, around a certain steady state using a localized interior control acting only in the concentration equation. To do it, we first linearize the nonlinear equation around the steady state. The linearized system turns out to be a system coupled between second order and fourth order parabolic equations and the control acts in the fourth order parabolic equation. The null controllability of the linearized system is obtained by a duality argument proving an observability inequality. To prove the observability inequality, a new Carleman inequality for the coupled system is derived. Next, using the source term method, it is shown that the null controllability of the linearized system with non-homogeneous terms persists provided the non-homogeneous terms 
satisfy certain estimates in a suitable weighted space. Finally, using a Banach fixed point theorem in a suitable weighted space, the local controllability of the nonlinear system is obtained. 
\end{abstract}

\maketitle


\section{Introduction}
The widely recognized Cahn-Hilliard-Navier-Stokes (CHNS) equation, introduced by Hohenberg and Halperin [\cite{H}, see also \cite{g96}], serves as a prominent model for describing the evolution of an incompressible isothermal mixture comprising two immiscible fluids. This model, also known as a diffuse interface model, replaces the sharp interface between the fluids with a diffuse interface by introducing an order parameter and takes dimensionless form in higher dimension as:
\begin{equation}\label{gen}
\left\{
\begin{aligned}
  \partial_t\phi + \mathbf{u} \cdot \nabla \mathrm{\phi} &= \text{div}(m(\mathrm{\phi})\nabla\mu), \, \, \text{ in } \Omega \times (0,T), \\
        \mu &= -\Delta\mathrm{\phi} + F'(\mathrm{\phi}), \\
        \partial_t\mathbf{u} - \text{div}(\nu(\mathrm{\phi}) \mathrm{D}\mathbf{u}) + (\mathbf{u}\cdot \nabla)\mathbf{u} + \nabla \pi &= \mu \nabla \mathrm{\phi}, \, \, \text{ in } \Omega \times (0,T), \\
        \text{div}~\mathbf{u} & = 0, \, \, \text{ in } \Omega \times (0,T), \\
         \mathbf{u} & = 0, \,\, \text{ on } \partial\Omega\times(0,T), \\
 \frac{\partial\mathrm{\phi}}{\partial\mathbf{n}} = 0, \,  \frac{\partial\mu}{\partial\mathbf{n}} & = 0, \,\, \text{ on } \partial\Omega\times(0,T), \\
       \mathbf{u}(0) = \mathbf{u}_0 ,\,\, \mathrm{\varphi}(0) & = \mathrm{\varphi}_0, \,\, \text{ in } \Omega,  
\end{aligned}   
\right.
\end{equation}
where $\mathbf {u}(t, \mathrm{x})$ is the average velocity of the fluid and $\mathrm{\phi}(t, \mathrm{x})$ is the relative concentration of the fluid. Here, $\Omega$ is a bounded domain in $\mathbb{R}^d, d\geq 2 $ with a sufficiently smooth boundary $\partial \Omega$. The density is taken as matched density, i.e., constant density, which is equal to 1. Moreover, $m$ is mobility of binary mixture, $\mu$ is a chemical potential, $\pi$ is the pressure, $\nu$ is the viscosity and $F$ is a double well potential. The symmetric part of the gradient of the flow velocity vector is denoted by $\mathrm{D}\mathbf{u}$,  that is, $\mathrm{D} \mathbf{u} $ is the strain tensor $\frac{1}{2}\left(\nabla\mathbf{u}+(\nabla \mathbf{u})^{\top}\right)$.\\
The existing theoretical literature, encompassing aspects such as well-posedness, strong solutions, optimal control problem, attractors and long-term behavior without being exhausted can be summarized in \cite{unique19, GT22, bbd24, yfc16, robust, convergence,dynamics,gal10asym,arma}.  The controllability aspects of this coupled  system has not been considered yet in the literature. Due to significance  of this diffuse interface model in the fluid dynamics, we are interested to study its controllability properties.

To gain insights into the problem, we choose to explore a one dimensional version of the Cahn-Hilliard-Navier-Stokes model, which is the Cahn-Hilliard-Burgers' model. 
In particular, we consider the  Cahn-Hilliard-Burgers' model with constant mobility $m(\phi)=1$ and constant viscosity $\nu(\phi)=\gamma>0$ in 
$Q_T=(0, T)\times (0, 1)$, for any finite time $T>0$, with a control $h$ acting in the concentration equation:
\begin{equation}\label{CHB}
\left\{
\begin{aligned}
&u_t-\gamma u_{xx}+uu_x = -\phi_{xx}\phi_x + F'(\phi)\phi_x + f_s\quad  \text{in } Q_T, \\
&\phi_t +u\phi_x+\phi_{xxxx}= (F'(\phi))_{xx} + \chi_{\mathcal{O}}h \quad  \text{in } Q_T,\\
& u(t,0) = 0, \quad u(t,1) = 0, \quad \quad t\in (0, T), \\
& \phi_x(t,0) = 0, \quad \phi_x(t,1) = 0, \quad \phi_{xxx}(t,0) = 0, \quad \phi_{xxx}(t,1) = 0, \quad t\in (0, T),\\
& u(0,x) = u_0(x), \quad \phi(0,x) = \phi_0(x), \quad x\in (0,1),
\end{aligned}   
\right.
\end{equation}
where $\chi$ is the characteristic function of an open set $\mathcal{O}\subset (0, 1)$ and $f_s\in L^2(0, 1)$ is a given function, and $F$ denotes the regular double well potential which takes the particular form 
\begin{align}
   F(r) = (r^2-1)^2, \quad r\in \mathbb{R}. \label{reg potential}
\end{align}
The main objective of this article is to study the local controllability of \eqref{CHB}--\eqref{reg potential} around $(\overline{u}, \overline{\mathrm{\phi}})$,
a steady state solution to \eqref{CHB} without control.

Before stating our main results, we first discuss the existence of a certain steady state around which the controllability of \eqref{CHB}--\eqref{reg potential} will be studied. 

\noindent 
Note that $(\overline{u}, \overline{\mathrm{\phi}})$, a steady state solution to \eqref{CHB} with $h=0$, satisfies 
\begin{equation}\label{steady CHB}
\left\{
\begin{aligned}
&-\gamma \overline{u}''+\overline{u}\, \overline{u}' = -\overline{\phi}''\overline{\phi}' + \left(4\bar \phi^3-4 \bar \phi\right)\overline{\phi}' + f_s \quad \text{in } (0, 1) , \\
& \overline{u}\overline{\phi}'+\overline{\phi}''''=24 \overline{\phi}\,\left(\overline{\phi}'\right)^2 + (12{\bar{\phi}^2}-4)\overline{\phi}''  \quad \text{in }   (0, 1),\\
& \overline{u}(0) = 0, \quad \overline{u}(1) = 0, \\
&\overline{\phi}'(0) = 0, \quad \overline{\phi}'(1) = 0,\\
&\overline{\phi}'''(0) = 0, \quad \overline{\phi}'''(1) = 0,
\end{aligned}   
\right.
\end{equation}
for the given $f_s\in L^2(0,1)$ as in \eqref{CHB}.

We consider the case where $\overline{\phi}$ is a constant and then $\overline{u}$ satisfies
\begin{equation}\label{heat equ}
\left\{
\begin{aligned}
&-\gamma \overline{u}''+\overline{u} \, \overline{u}' = f_s,\\
& \overline{u}(0) = 0, \quad \overline{u}(1) = 0.
\end{aligned}   
\right.
\end{equation}
Under the condition $\|f_s\|_{L^2(0,1)}$ is small enough, using a fixed point argument, it can be shown that \eqref{heat equ} admits a solution $\overline{u}$ in $H^2(0,1)\cap H^1_0(0,1)$. Further, using continuous embedding result $H^2(0, 1)\hookrightarrow C^1([0, 1])$ for one-dimensional space, we get $\overline{u}\in C^1([0, 1])$.

\begin{remark}\label{rem-assump-steadystate}
In this article, we consider $(\overline{u}, \overline{\phi})$, where $\overline{\phi}$ is a constant but is not equal to $0, 1, -1$, and 
$\overline{u}\in C^1([0, 1])\cap H^2(0,1)\cap H^1_0(0,1)$ satisfies 
\eqref{heat equ} with a given $f_s\in L^2(0,1)$ with $\|f_s\|_{L^2(0,1)}$ is small enough. 
\end{remark}

Our goal is to find a control $h$ such that the solution of the nonlinear system \eqref{CHB} can be driven to the given steady state $(\overline{u}, \overline{\phi})$ at any given time $T>0$. In other words, for any given $T>0$, we want to find a control $h$ such that $(u-\overline{u}, \phi-\overline{\phi})$, where $(u, \phi)$ satisfies \eqref{CHB}--\eqref{reg potential}, reaches at $(0, 0)$ at time $T$. 

For that, we set $w=u-\overline{u}, \, \psi=\phi-\overline{\phi}$, where $(w, \psi)$  satisfies 
\begin{equation}\label{non lin}
		\begin{cases}
			w_t-\gamma w_{xx}+\bar u(x) w_x+\bar u'(x) w=\gamma_1 \psi_x+N_1(w,\psi), & \text{in } Q_T, \\
			\psi_t+ \psi_{xxxx}+\gamma_2 \psi_{xx} + \bar u(x) \psi_x = N_2(w,\psi)+\chi_{\mathcal{O}} h,  & \text{in } Q_T,\\
			w(t, 0)=0,  \  w(t, 1)=0,  & t \in (0, T), \\
			\psi_x(t,0)=0, \  \psi_x(t,1)=0, & t \in (0, T),    \\
			\psi_{xxx}(t, 0)=0,  \  \psi_{xxx}(t, 1)=0, & t \in (0, T),\\
			w(0, x)=w_0(x),\, \psi(0,x)=\psi_0(x), & x \in  (0, 1),
		\end{cases}
	\end{equation}
where $ ( w_0(x),\,\psi_0(x) ) =  ( u_0(x)-\overline{u}(x), \phi_0(x)-\overline{\phi}) $ and \begin{equation}\label{nonlinear}
		\begin{cases}
			N_1(w,\psi)=-ww_x-\psi_x \psi_{xx}+4\psi^3\psi_x+12\bar\phi\psi^2\psi_x+12\bar \phi^2\psi\psi_x-4\psi\psi_x,\\
			N_2(w,\psi)=-w\psi_x+12\psi^2\psi_{xx}+24\bar \phi\psi \psi_{xx}+24\psi \psi_x^2+24\bar \phi \psi_x^2,
		\end{cases}
	\end{equation}
where 
\begin{equation}\label{eqconstant}
 \gamma_1 = 4\bar \phi^3-4 \bar \phi, \quad \gamma_2= -(12{\bar{\phi}^2}-4).
 \end{equation}

Following the  classical strategy to establish the controllability of non-linear problem \eqref{non lin}--\eqref{nonlinear}, we first prove the controllability of its linear counterpart.
Hence let us introduce the linearized system of \eqref{CHB} around the steady state $(\overline{u}, \overline{\phi})$ by collecting the linear terms in \eqref{non lin} as
\begin{equation}
		\begin{cases}
			\label{CH}
			w_t-\gamma w_{xx}+\bar u(x) w_x+\bar u'(x) w=\gamma_1 \psi_x, & \text{in } Q_T, \\
			\psi_t+ \psi_{xxxx}+\gamma_2 \psi_{xx} + \bar u(x) \psi_x = \chi_{\mathcal{O}} h, &\text{in } Q_T,
				\end{cases} \end{equation}
with boundary conditions
\begin{equation}
	\begin{cases}\label{bd}
		w(t, 0)=0,  \  w(t, 1)=0,  & t \in (0, T), \\
		\psi_x(t,0)=0, \  \psi_x(t,1)=0, & t \in (0, T),    \\
		\psi_{xxx}(t, 0)=0,  \  \psi_{xxx}(t, 1)=0, & t \in (0, T),\\
	\end{cases}
\end{equation}
and with initial conditions 
\begin{equation}\label{in}
	w(0, x)=w_0(x), \psi(0,x)=\psi_0(x),\quad  x \in  (0, 1). 
\end{equation}

{
\begin{remark}
Note that due to the condition that $\overline{\phi}$ is not equal to $0, 1, -1$ as mentioned in \Cref{rem-assump-steadystate}, we have $\gamma_1\neq 0$ and so, the system $\eqref{CH}_1$ is coupled with $\eqref{CH}_2$. Thus, it is meaningful to explore the controllability of the above system using a control acting only in $\eqref{CH}_2$. 
\end{remark}
}

For any given $T>0$ and $\mathcal{O}$ an open subset in $(0, 1)$, we say that the linear system  \eqref{CH}--\eqref{in} is null controllable in $(L^2(0, 1))^2$ at time $T$ using a control localized in $\mathcal{O}$, if for every initial data $(w_0,\psi_0) \in (L^2(0, 1))^2$ there exists a control $h \in L^2(0, T; L^2(\mathcal{O}))$ such that 
$(w, \psi)$, the solution of \eqref{CH}--\eqref{in} satisfies 
$$\left(w(T, \cdot), \,  \psi(T, \cdot)\right)=(0,0), \quad \mathrm{in}\quad (L^2(0, 1))^2.$$

We define the space 
\begin{align}\label{eqspace}
    H^2_1:= \{\varphi\in \mathrm{H}^2(0, 1): \varphi'(0) = 0 = \varphi'(1)\}
\end{align}
which is a subspace of $H^2(0, 1)$ endowed with the norm
$$ \|\varphi\|_{H^2_1}:= \|\varphi\|_{L^2}+ \|\varphi''\|_{L^2}, \quad \forall\, \varphi\in H^2_1.$$
Moreover, from \cite[p. 271, Remark 8]{Brezis}, we recall that for any $\epsilon>0$, there exists a positive constant $K(\epsilon)$ depending on $\epsilon$ such that  
\begin{equation}\label{ineq1}
\|\varphi'\|_{L^2} \le \epsilon \|\varphi''\|_{L^2} + K(\epsilon)\|\varphi\|_{L^2}, \quad \forall\, \varphi\in H^2(\Omega).
\end{equation}

Our  main results in the article are stated below. First, we have following null controllability result for the linearized system \eqref{CH}--\eqref{in}: 
\begin{theorem}\label{Thm-linear}
Let $(\overline{u}, \overline{\phi})$ be as mentioned in \Cref{rem-assump-steadystate} and let $\mathcal{O}$ be any open subset of $(0, 1)$.
For any time $T>0$ and any initial data $(w_0,\psi_0)\in  (L^2(0,1))^2,$ there exists a control $h\in L^2(0, T; L^2(\mathcal{O}))$ such that the system \eqref{CH}--\eqref{in} admits a unique solution $(w, \psi)\in C([0,T]; L^2(0, 1))\cap L^2(0, T; H^1_0(0, 1))\times C([0, T]; L^2(0, 1))\cap L^2(0, T; H^2_1)$ satisfying 
$$\left(w(T,\cdot), \,  \psi(T,\cdot)\right)=(0,0) \quad \mathrm{in}\quad  (L^2(0,1))^2.$$ 
Moreover, the  control function $h$ has the following estimate:
\begin{align}\label{control_estimate}
	\|h\|_{L^2(0, T; L^2(\mathcal{O}))} \leq M e^{M(T +\frac{1}{T^m})} \|(w_0,\psi_0)\|_{(L^2(0,1))^2},
	\end{align}
where the constant $M>0$  is independent of $T$ and $(w_0,\psi_0)$ and $m>3$ be a natural number. 
\end{theorem}

Further using a fixed point argument, we obtain the following result for the non-linear system:
\begin{theorem}\label{Thm-nonlinear} 
Let $(\overline{u}, \overline{\phi})$ be as mentioned in \Cref{rem-assump-steadystate} and let $\mathcal{O}$ be any open subset of $(0, 1)$. For any $T>0$, there exists a positive number $\mu$ such that for any $(w_0,\psi_0)\in  (L^2(0,1))^2,$ with $\norm{(w_0,\psi_0)}_{L^2\times L^2}\leq \mu,$ 
there exists a control $h\in L^2(0, T; L^2(\mathcal{O}))$ such that the system \eqref{non lin} admits a solution $(w, \psi)\in C([0,T]; L^2(0, 1))\cap L^2(0, T; H^1_0(0, 1))\times C([0, T]; L^2(0, 1))\cap L^2(0, T; H^2_1)$ satisfying 
$$\left(w(T,\cdot), \,  \psi(T,\cdot)\right)=(0,0) \quad \mathrm{in}\quad  (L^2(0,1))^2.$$ 
\end{theorem}

The proof of the interior null controllability of the linearized system \eqref{CH}--\eqref{in} in \Cref{Thm-linear} relies on proving an observability inequality for the solutions of the adjoint system. To prove the observability inequality, leveraging global Carleman estimates for the Cahn-Hilliard equation and Burgers' equation, we first derive a unified Carleman estimate for our coupled system. This, in turn, facilitates the proof of null-controllability for the linearized system with an appropriate control cost of {$Me^{M(T+\frac{1}{T^m})}$}. This step is pivotal for deducing local controllability results for the nonlinear model. Subsequently, we employ the source term method, specifically proving the null-controllability of the linear model \eqref{CH}--\eqref{in} with suitably chosen source terms from a weighted Hilbert space where the weight function blows up as $t$ approaches $T^-$. 
Finally, employing the Banach fixed-point argument in a suitable weighted space, we establish local null-controllability for the nonlinear model and hence \Cref{Thm-nonlinear}.

Let us mention some related works in this direction from the available literature. The controllability of parabolic equations using Carleman inequality has been studied in \cite{MR2257228} in details. The fourth order parabolic equation, namely Kuramoto-Sivashinsky (KS) equation, has been extensively appears in many works such as  \cite{CE11}, \cite{Zhou}, \cite{Cer17}, \cite{Taka17}.  In particular, in \cite{Cer17}, the authors proved that the linear KS equation with Neumann boundary conditions is null-controllable by a  control acting in some open subset of the domain. In \cite{Taka17}, the author studied the boundary local null-controllability of the KS equation by utilizing the source term method (see \cite{Tucsnak-nonlinear}) followed by  the Banach fixed point argument where a suitable control cost $C e^{C/T}$ of the linearized model plays  the crucial role. Similar strategy has  been applied in \cite{hernandezsantamaria:hal-03090716} to study the boundary local null-controllability of a simplified stabilized KS system. 

To include dispersive and extra dissipative effects in the KS equation, a coupled system containing fourth (Kuramoto-Sivashinsky-Korteweg-de Vries, in short KS-KdV equation) and second order parabolic (heat) equations, under the name of stabilized Kuramoto-Sivashinsky system has appeared for instance in \cite{PhysRevE.64.046304}. The controllability of such system has been considered   in several works, namely  \cite{CE10}, \cite{CE12}, \cite{Cerpa-Mercado-Pazoto}, \cite{CE16}, \cite{kumar2022null}.   E. Cerpa, A. Mercado and A. F. Pazoto in \cite{Cerpa-Mercado-Pazoto} demonstrated the  local null-controllability of Stabilized KS equation by means of localized interior control acting in the KS-KdV equation. Then,  E. Cerpa and N. Carre\~{n}o \cite{CE16} proved the local null-controllability to the trajectories with a localized interior control exerted in the heat equation. Two different types of Carleman estimates for the linearized model followed by the inverse mapping theorem have been implemented to  conclude their controllability results. We also point out that recently, the controllability of KS-KdV-transport model and KS-KdV-elliptic equation have been discussed in \cite{majumdar:hal-03695906}, \cite{BM23} respectively.

The spirit of our work  relies on that of \cite{Guzman, Cerpa-Mercado-Pazoto}. The fourth order parabolic equation of our system is similar to that of the equation considered in \cite{Guzman} with the same boundary conditions. However, unlike our case, \cite{Guzman} deals with only the fourth order parabolic equation and not a coupled system between fourth order and second order parabolic equations. Rather, in \cite{Cerpa-Mercado-Pazoto}, a system coupled between a fourth order parabolic equation and a second order parabolic equation is considered. But the boundary condition for the fourth order equation is different than us. Thus, though the main strategy used to prove our result is classical and is similar to that of \cite{Guzman, Cerpa-Mercado-Pazoto}, due to the presence of the coupling and the different boundary conditions, our results do not follow directly from any of the earlier works mentioned above. In fact, to handle the coupling between equations and the boundary conditions, we need to derive a new Carleman inequality to prove the observability inequality. 

The novelty of our work is to deal with the coupling between second  and fourth order parabolic equations with the boundary conditions and nonlinear terms and to obtain the controllability results. Even though the main strategy of the proof of our results is standard, each step of the proof has to be investigated due to the difficulty posed by the coupling and the boundary conditions and the nonlinear terms. For instance, a new Carleman inequality for the coupled system has to be derived for this system with a careful study on the weight functions applicable for both the second order and the forth order parabolic equations with the boundary conditions. To our knowledge, this is the first attempt to study the controllability of the Cahn-Hilliard-Navier Stokes equation around a trajectory. Our results obtained here  give an insight how to approach to study the controllability of the Cahn-Hilliard-Navier Stokes equation in more general set-up and in higher dimensions (see \Cref{Sec: remarks}). 

{
The paper is organized as follows. In \Cref{sec:wellposed}, we first study the well-posedness of the linear coupled system and its associated adjoint problem. 
\Cref{Sec-carleman-1} is devoted to prove a Carleman inequality associated to the solution of the adjoint problem. This is the main tool to prove the null controllability result for the linear coupled system in \Cref{sec-car1}. First in \Cref{sec-obs}, an observability inequality for the solution of the adjoint problem is established and consequently in \Cref{sec-proof}, \Cref{Thm-linear} is proved. Next, in \Cref{Sec:control with source}, the null controllability of the 
linear coupled system with non-homogeneous terms from a suitable weighted space is proved. Finally, using a Banach fixed point argument, \Cref{Thm-nonlinear}, the local null controllability of \eqref{non lin} is proved in \Cref{Sec:Control nonlin}. At the end, in \Cref{Sec: remarks}, we mention a few open problems in this direction and possible extensions  of our results.} 


\section{Well-posedness of the linear system}\label{sec:wellposed}
 In this section we analyze the well-posedness of linearized system with given forcing terms $f_1, f_2$
 \begin{equation}\label{LinSystem}
 	\begin{cases}
 		w_t-\gamma w_{xx}+\bar u(x) w_x+\bar u'(x) w=\gamma_1 \psi_x+f_1, & \text{in } Q_T, \\
 		\psi_t+ \psi_{xxxx}+\gamma_2 \psi_{xx} + \bar u(x) \psi_x = f_2, &\text{in } Q_T\\
 		w(t, 0)=0,  \  w(t, 1)=0,  & t \in (0, T), \\
 		\psi_x(t,0)=0, \  \psi_x(t,1)=0, & t \in (0, T),    \\
 		\psi_{xxx}(t, 0)=0,  \  \psi_{xxx}(t, 1)=0, & t \in (0, T),\\
 		w(0, x)=w_0(x),\, \psi(0,x)=\psi_0(x), & x \in  (0, 1). 
 	\end{cases}
 \end{equation}
 
To study the above system, first, we consider the equation $\eqref{LinSystem}_2$ with the associated boundary conditions and initial data
 \begin{equation}\label{eqbiharmonic}
 \begin{array}{l}
 \psi_t+ \psi_{xxxx}+\gamma_2 \psi_{xx} + \bar u(x) \psi_x =f_2, \quad \text{in }\quad Q_T,\\
 \psi_x(t,0)=0, \  \psi_x(t,1)=0, \quad \psi_{xxx}(t, 0)=0,  \  \psi_{xxx}(t, 1)=0, \quad  t \in (0, T),\\
 \psi(0, \cdot)= \psi_0(\cdot), \quad \text{in} \quad (0,1).
 \end{array}
 \end{equation}
 The above equation can be written in the operator form: $\psi_t(t) = A\psi(t)+ f_2(t) \quad \forall\, t\in (0, T), \quad \psi(0)=\psi_0$, where 
   \begin{align}\label{A}
   D(A) = \{\psi\in\mathrm{H}^4(0,1):\psi'(0)=0=\psi'(1), \psi'''(0)=0=\psi'''(1)\},\\ \nonumber
   A\psi = -\psi''''-\overline{u}(x)\psi' -\gamma_2\psi'', \quad \psi\in D(A).
   \end{align}
   
 \begin{lemma}\label{lembiharmonic}
 The operator $(A, D(A))$ defined in \eqref{A} forms a strongly continuous semigroup on $L^2(0,1)$. 
 Further, for any $\psi_0\in L^2(0, 1)$ and any $f_2\in L^2(0, T; L^2(0, 1))$, \eqref{eqbiharmonic} admits a unique solution $\psi\in C([0, T]; L^2(0, 1))\cap L^2(0, T; H^2_1)$. 
 \end{lemma}
\begin{proof}
We show that $(A, D(A))$ is quasi-dissipative and maximal operator in $L^2(0, 1)$. 

\noindent
\textbf{Step 1.} Quasi-dissipative: Considering the Hilbert space $L^2(0,1)$ with real field, using the properties of $\overline{u}$ in Remark \ref{rem-assump-steadystate}, note that for all $\psi\in D(A)$, 
\begin{align*}
       \langle A\psi, \psi\rangle_{L^2}&=-\int_0^1|\psi''(x)|^2dx- \int_0^1\overline{u}(x)\psi(x) \psi'(x) dx- \gamma_2\int_0^1\psi''(x) \psi(x) dx\\
       \leq& -\int_0^1|\psi''(x)|^2 dx + \frac{1}{2}\int_0^1\overline{u}'(x)\psi^2(x) dx +\frac{1}{2}\int_0^1|\psi''(x)|^2 dx + \frac{\gamma_2^2}{2}\int_0^1|\psi(x)|^2 dx\\
       \leq & -\frac{1}{2}\int_0^1|\psi''(x)|^2 dx +\Big(\frac{\|\overline{u}'\|_{\mathrm{L}^\infty}}{2}+\frac{\gamma_2^2}{2}\Big)\int_0^1|\psi(x)|^2 dx
       \leq  \beta \|\psi\|^2_{L^2(0,1)},
   \end{align*}
for some $\beta=\Big(\frac{\|\overline{u}'\|_{\mathrm{L}^\infty}}{2}+\frac{\gamma_2^2}{2}\Big)>0$. Hence $(A, D(A))$ is quasi-dissipative in $L^2(0, 1)$.

\noindent 
\textbf{Step 2.} Maximality: We claim that for some $\lambda_0>\beta$ large enough, $(\lambda_0 I-A):D(A)\rightarrow L^2(0, 1)$ is onto. To prove it, we need to show that for any $f\in L^2(0, 1)$, there exists $\psi\in D(A)$ satisfying 
   \begin{align}\label{maximal}
       \lambda_0\psi+\psi''''+\overline{u}\psi'+\gamma_2\psi'' =f,\\
       \psi'(0)=0=\psi'(1), \quad \psi'''(0)=0=\psi'''(1).
   \end{align}
Let us define the bilinear form $a : H^2_1\times H^2_1 \rightarrow \mathbb{R}$ by
\begin{align}
    a(\psi, \varphi) =\lambda_0\int_0^1\psi\varphi+\int_0^1\psi''\varphi''+\int_0^1\overline{u}\psi'\varphi+\gamma_2\int_0^1\psi''\varphi, \quad \forall\, \psi, \varphi\in  H^2_1.
\end{align}
Clearly, $a(\cdot, \cdot)$ is continuous and further, for any $\psi\in H^2_1$, we get 
\begin{align*}
   a(\psi, \psi)  &\geq \frac{1}{2}\int_0^1|\psi''(x)|^2 dx +\Big(\lambda_0-\beta\Big)\int_0^1|\psi(x)|^2 dx \geq C\|\psi\|^2_{H^2_1},
\end{align*}
for some positive constant $C$, provided $\lambda_0>\beta$ and thus the bilinear form $a(\cdot, \cdot)$ is coercive. 
Using Lax-Milgram lemma, (for example, see \cite[Corollary 3,  Chapter VI]{DLv5}), for any $f\in L^2(\Omega)$, there exists a unique $\psi\in H^2_1$ so that 
\begin{align}\label{variation of A}
    a(\psi,\varphi) = (f, \varphi)_{L^2}, \quad \forall\, \varphi\in H^2_1. 
\end{align}
Now, using elliptic regularity result, it can be derived $\psi\in D(A)$ and thus, $(\lambda_0 I-A):D(A)\rightarrow L^2(0, 1)$ is onto. 

Hence, \cite[Theorem 12.22]{RR} gives that $(A, D(A))$ generates strongly continuous semigroup on $L^2(0, 1)$ and \eqref{eqbiharmonic} admits a unique solution $\psi\in C([0, T]; L^2(0, 1))$. From \cite[Theorem 11.3]{RR}, the regularity results for parabolic equation give that solution of \eqref{eqbiharmonic} belongs to $\psi\in C([0, T]; L^2(0, 1))\cap L^2(0, T; H^2_1)$. 
\end{proof}

We have the following well-posedness result for linearized system.   
\begin{theorem}\label{thwellposed}
Let $(w_0, \psi_0) \in (L^2(0,1))^2$ and $f_1, f_2 \in L^2(0, T; L^2(0, 1))$. Then the system \eqref{LinSystem} has a unique solution $(w, \psi)\in C([0,T];L^2(0,1))\cap L^2(0, T; H^1_0(0,1)) \times C([0, T]; L^2(0, 1))\cap L^2(0, T; H^2_1),$ where the space $H^2_1$ is defined in \eqref{eqspace}. 

\noindent 
Further, the solution of \eqref{CH}--\eqref{bd} satisfies
\begin{align}\label{energy}
\|w\|_{C([0,T]; L^2(0,1))\cap L^2(0,T; H^1_0(0,1))} \leq Ce^{CT}\left(\|w_0\|_{L^2(0, 1)} + \|\psi_x\|_{L^2(0, T; L^2(0, 1))}+\|f_1\|_{L^2(0, T; L^2(0, 1))}\right),  \\
 \|\psi\|_{C([0,T]; L^2(0,1))} + \|\psi\|_{L^2(0, T; H^2(0,1))} \leq Ce^{CT}\left( \|\psi_0\|_{L^2(0,1)}+ \|f_2\|_{L^2(0, T; L^2(0, 1))}\right), \nonumber
\end{align}
for some positive constant $C$. 
\end{theorem}
\begin{proof}
\noindent
\textbf{Existence and uniqueness of solution.} We first solve $\eqref{CH}_2$ with the corresponding boundary conditions and initial data and then treating the right-hand side of the first equation of \eqref{CH} as forcing term, we solve the heat equation with Dirichlet boundary condition.\\
\noindent 
For any $\psi_0\in L^2(0, 1)$ and $h \in L^2(0, T; L^2(0, 1))$, from Lemma \ref{lembiharmonic}, it follows that there exists a unique solution 
$\psi\in C([0, T]; L^2(0, 1))\cap L^2(0, T; H^2_1)$ to the equation $\eqref{LinSystem}_2$ with the boundary conditions $\eqref{LinSystem}_4- \eqref{LinSystem}_5$. Now putting this $\psi$ in $\eqref{LinSystem}_1$, we get the heat equation 
\begin{equation}\label{heat part}
\left\{
\begin{aligned}
& w_t -\gamma w_{xx} + \overline{u}w_x + \overline{u}' w= \gamma_1\psi_x , \quad (t,x)\in (0,1)\times(0,T),\\
& w(t,0) = 0, \quad w(t,1) = 0, \quad t\in (0, T),\\
&w(0,x) = w_0(x), \quad x\in (0,1).
\end{aligned}   
\right.
\end{equation}
Since $\psi_x\in L^2(0, T; L^2(0,1))$, for any $w_0\in L^2(0, 1)$ and $f_1\in L^2(0, T; L^2(0,1))$, \eqref{heat part} admits a unique solution $w\in C([0,T];L^2(0,1))\cap L^2(0, T; H^1_0(0,1))$ (See, for example, \cite[Theorem 11.3, p. 382]{RR}). Thus, the first part of the theorem is proved.

\noindent 
\textbf{Energy estimates.} To track the dependency of the constant on time $T$, we give the outline of the estimates below. First we consider the equations with smooth initial data and smooth right hand side. Multiplying the equation $\eqref{CH}_2$ by $\psi$ and integrating by parts we get
\begin{align*}
    \frac{1}{2}\frac{d}{dt}\int_0^1|\psi(t,x)|^2\,dx + \int_0^1 |\psi_{xx}(t, x)|^2\, dx =& \frac{1}{2}\int_0^1\overline{u}'(x)|\psi(t,x)|^2\,dx -\gamma_2\int_0^1\psi_{xx}(t, x)\psi(t, x)\, dx \\ \nonumber&\qquad+ \int_0^1 f_2(t, x)\psi(t, x)\, dx,
\end{align*} 
and then applying Young's inequality, choosing any small $0<\epsilon<1$, 
\begin{align}
 \frac{d}{dt}\int_0^1|\psi(t,x)|^2\, dx+ 2(1-\epsilon)\int_0^1 | \psi_{xx}(t,x)|^2\, dx 
 &\leq \bigg(\|\overline{u}'\|_{\mathrm{L}^\infty}+2C_1(\epsilon)+1\bigg)\int_0^1|\psi(t, x)|^2\, dx\\ \nonumber&\qquad +\int_0^1 |f_2(t, x)|^2\, dx.\label{int_ineq}
\end{align}
Denoting $\beta' =\|\overline{u}'\|_{\mathrm{L}^\infty}+2C_1(\epsilon)+1$, and from above considering 
\begin{align*}
     \frac{d}{dt}\int_0^1|\psi(t, x)|^2\, dx \leq\beta'\int_0^1|\psi(t, x)|^2\, dx +\int_0^1 |f_2(t, x)|^2\, dx,
\end{align*}
the Gronwall's inequality gives
\begin{align*}
   & \int_0^1|\psi(t, x)|^2\, dx \leq e^{\beta't}\left(\|\psi(0)\|^2_{L^2(0,1)}+\int_0^t\int_0^1 |f_2(t, x)|^2\, dx\, dt\right), \quad \forall\, t\in [0, T],\\
  &  \sup_{t\in [0, T]}\|\psi(t, \cdot)\|^2_{L^2(0, 1)}\leq e^{\beta'T}\left(\|\psi(0)\|^2_{L^2(0,1)}+ \|f_2\|^2_{L^2(0, T; L^2(0, 1)}\right).
\end{align*}
Further, from \eqref{int_ineq}, we have 
$$
\begin{array}{l}
\|\psi_{xx}\|^2_{L^2(0, T; L^2(0, 1))}\le \beta'\|\psi\|^2_{L^2(0, T; L^2(0, 1))}+\|f_2\|^2_{L^2(0, T; L^2(0, 1))}\\
   \le \beta'T e^{\beta'T}\left(\|\psi(0)\|^2_{L^2(0,1)}+ \|f_2\|^2_{L^2(0, T; L^2(0, 1))}\right)+ \|f_2\|^2_{L^2(0, T; L^2(0, 1))}\\
   \le C e^{CT}\left(\|\psi(0)\|^2_{L^2(0,1)}+ \|f_2\|^2_{L^2(0, T; L^2(0, 1))}\right),
\end{array}
$$
for some $C>0$ independent of $T$ and thus, $\eqref{energy}_2$ is proved for smooth data. 

Now, multiplying the equation $\eqref{CH}_1$ by $w$ and integrating by parts we obtain
\begin{align}
&\frac{1}{2}\frac{d}{dt}\int_0^1 |w(t, x)|^2\, dx + \gamma\int_0^1 |w_x(t, x)|^2\, \\ \nonumber&= \gamma_1 \int_0^1 \psi_x(t, x) w(t, x)\, dx+ \int_0^1 f_1(t, x) w(t, x)\, dx -\frac{1}{2}\int_0^1\overline{u}'(x)|w(t, x)|^2\, dx\nonumber\\
& \leq \frac{1}{2}\left(\gamma_1\int_0^1|\psi_x(t, x)|^2\, dx+ \int_0^1|f_1(t, x)|^2\, dx\right) + C\int_0^1|w(t, x)|^2\, dx, \nonumber
\end{align}
for some $C>0$ independent of $T$, and following similar arguments as above, we derive the estimate $\eqref{energy}_1$ for smooth data. 
\noindent
Using density argument, we obtain the estimates for any initial data from $L^2(0,1))^2$ and right hand side from $L^2(0, T; L^2(0, 1))$.
\end{proof}

For later purpose, we state the following result for the case when the source terms $f_1, f_2 \in L^1(0,T; L^2(0,1))$. 
\begin{proposition}\label{prop-esti-regular}
For any given initial data $(w_0,\psi_0) \in (L^2(0,1))^2$, source terms $f_1,f_2 \in L^1(0,T; L^2(0,1))$,  there exists a unique { mild} solution $(w,\psi)$ of \eqref{LinSystem} such that
$w\in C([0,T]; L^2(0,1)) \cap L^2(0,T; H^1_0(0,1))$ and $\psi\in C([0,T]; L^2(0,1)) \cap L^2(0,T; H^2_1)$. 
In addition, there exists a positive constant $C$, independent in $T$ such that 
	\begin{multline*}
		\left\|{w} \right\|_{C([0,T]; L^2(0,1)) \cap L^2(0,T; H^1_0(0,1))} + 	\left\|\psi \right\|_{C([0,T]; L^2(0,1)) \cap L^2(0,T; H^2_1)} \\
		\leq  Ce^{CT}  \left(\|w_0\|_{L^2(0,1)} +\|\psi_0\|_{L^2(0,1)}+ \left\|(f_1,f_2) \right\|_{L^1(L^2\times L^2)}\right).
	\end{multline*}
\end{proposition}

The adjoint system corresponding to the system \eqref{CH}--\eqref{bd}--\eqref{in} is given by: 
 \begin{equation}
 	\begin{cases}
 		\label{CH adj}
 		\sigma_t+\gamma \sigma_{xx}+\bar u(x) \sigma_x=0, & \text{ in } Q_T, \\
 		v_t-v_{xxxx}-\gamma_2 v_{xx} + \bar u(x) v_x+\bar u'(x) v =\gamma_1\sigma_x &\text{ in } Q_T,\\
 		\sigma(t, 0)=0,  \  \sigma(t, 1)=0,  & t \in (0, T), \\
 		v_x(t,0)=0, \  v_x(t,1)=0, & t \in (0, T),    \\
 		v_{xxx}(t, 0)=0,  \  v_{xxx}(t, 1)=0, & t \in (0, T),\\
 		\sigma(T, x)=\sigma_T(x), v(T,x)=v_T(x) & x \in  (0, 1).
 	\end{cases} 
 \end{equation}
Now we state the well-posedness result for the adjoint system \eqref{CH adj}. The proof of the theorem is similar to that of Theorem \ref{thwellposed}. 

\begin{theorem}
For each $(\sigma_T, v_T) \in (L^2(0, 1))^2,$ the system \eqref{CH adj} has a unique solution $(\sigma, v)\in C([0,T];L^2(0,1))\cap L^2(0, T; H^1_0(0,1)) \times C([0, T]; L^2(0, 1))\cap L^2(0, T; H^2_1).$ Moreover there exists a positive constant $C$, depending on $T$, such that
  \begin{align}\label{adj_est}
      \|(\sigma, v)\|_{C([0,T];L^2(0,1))\cap L^2(0, T; H^1_0(0,1)) \times C([0, T]; L^2(0, 1))\cap L^2(0, T; H^2_1)} \leq C\|(\sigma_T, v_T)\|_{(L^2(0, 1))^2}.
  \end{align}
\end{theorem}

\section{Carleman inequality}\label{Sec-carleman-1}
This section is devoted to the derivation of suitable Carleman estimate for the adjoint system \eqref{CH adj} to deduce the required observability inequality.
To do so, let us define the some useful weight functions.
\paragraph{\bf Weight function}
Consider a non-empty open set $\mathcal{O}_0\subset \subset  \mathcal{O}$.
There exists a function $\nu \in C^4([0,1])$ such that  
\begin{align}\label{definition_nu} 
	\begin{cases}
		\nu(x)>0 \qquad \forall x\in (0,1), \ \ \nu(0)=\nu(1)=0,\, \nu''(0)=\nu''(1)=0. \\
		|\nu^\prime(x) | \geq \widehat c >0 \ \  \forall x \in \overline{(0,1) \setminus \mathcal{O}_0} \ \ \text{ for some } \widehat c > 0.
	\end{cases}
\end{align} 
It is clear that $\nu^{\prime}(0)>0$ and $\nu^\prime(1)<0$. We refer \cite{Guzman} where the existence of such function has been addressed.  

Now, for any constant $k,m\in \mathbb{N}, \text{ with } k>m>3$ and parameter $\lambda>1$,  we define the weight functions 
\begin{align}\label{weight_function}
	\vphi_m(t,x)= \frac{e^{\frac{m+1}{m}\lambda k \|\nu\|_{\infty}   }-  e^{\lambda\big( k\|\nu\|_{\infty} + \nu(x) \big)}}{t^m(T-t)^m}	, \quad \xi_m(t,x) = \frac{e^{\lambda\big( k\|\nu\|_{\infty} + \nu(x) \big)}}{t^m(T-t)^m}, \quad \forall \, (t,x) \in Q_T.
\end{align}
We have defined the above  weight functions by adopting   the works  \cite{G07}. It is clear that both $\vphi_m$ and $\xi_m$ are positive functions in $Q_T$. 

In further calculations we will need estimates on the time and higher order space derivatives of the weight functions. The estimates are given below. 
\begin{itemize}
\item  Using the definition of the weight functions we observe:
\begin{align}
	\	\xi_m^{1/m} &\leq T^{2m-2} \xi_m, \quad (t, x)\in Q_T. \label{esti_time_deri}
	\end{align}
 
	\item  For any $l\in \mathbb N$ and $s>0$, there exists a positive constant $C$ independent of $\lambda, s, t, T$ such that
	\begin{align}\label{weight-deri-t}
		\big|\partial_t(e^{-2s\vphi_m}\xi_m^l) \big| \leq C\left(s T+ T^{2m+1}\right) \xi_m^{(1+\frac{1}{m})} e^{-2s\vphi_m}\xi_m^{l}, \quad (t, x)\in Q_T.
	\end{align}
Using  \eqref{esti_time_deri} in \eqref{weight-deri-t}, we get that for some positive constant $C$ independent of $\lambda, s, t, T,$ 
	\begin{align}\label{weight_estimate-t-m}
		|\partial_t(e^{-2s\vphi_m} \xi^l_m)| \leq C\left(sT^{2m-1}+ T^{2m}T^{2m-1}\right) \xi^2_m (e^{-2s\vphi_m} \xi^l_m), \quad (t, x)\in Q_T.
	\end{align}

	\item  For any $(l,n)\in \mathbb N\times \mathbb N$ and $s>0$, there exists a positive constant $C$ independent of $\lambda, s, t, T$ such that
\begin{align}\label{weight-deri-x}
	\big|{\partial_x^{(n)}}(e^{-2s\vphi_m}\xi_m^l) \big| \leq C s^n \lambda^n e^{-2s\vphi_m}\xi_m^{l+n}, \quad (t, x)\in Q_T.
\end{align}  

\end{itemize}

\paragraph{\bf Some useful notations} We also  declare the following notations which will simplify the expressions of our Carleman inequalities.

\begin{itemize} 
		\item[--] For any function $\sigma\in L^2(0,T;H^3(0,1)\cap H^1_0(0,1))\cap H^1(0,T; H^1(0,1))$ and positive parameters $s$, $\lambda$, we denote 
	\begin{align}\label{Notation_elliptic}
		\mc{I}_H(s,\lambda; \sigma): = s\lambda^2  \iintQ e^{-2s\vphi_m} \xi_m |\sigma_{xx}|^2 + s^3\lambda^4 \iintQ e^{-2s\vphi_m} \xi^3_m |\sigma_{x}|^2. 
	\end{align}

	\item[--]  For any $v\in  L^2(0,T; H^4(0,1))\cap H^1(0,T; L^2(0,1))$ and positive parameters $s$, $\lambda$,  we denote
	\begin{multline}\label{Notation_KS}
		\mc{I}_{CH}(s,\lambda;v):=s^7\lambda^8 \iintQ e^{-2s\vphi_m} \xi_m^7 |v|^2 + s^5\lambda^6 \iintQ e^{-2s\vphi_m} \xi_m^5 |v_{x}|^2 +  s^3\lambda^4 \iintQ e^{-2s\vphi_m} \xi_m^3 |v_{xx}|^2 \\
		+ s\lambda^2 \iintQ e^{-2s\vphi_m} \xi_m |v_{xxx}|^2   +	s^{-1} \iintQ e^{-2s\vphi_m}\xi_m^{-1} (|v_t|^2+ |v_{xxxx}|^2\big). 
	\end{multline}

\end{itemize}

\paragraph{\bf Some standard Carleman estimates} 

Let us first recall the Carleman inequality for heat equation \cite{G07}. Let us define  
\begin{equation}\label{rev_wght}
\vphi_m^*(t) = \max_{x \in [0,1]}  \vphi_m(t,x) \quad \text{and} \quad \xi_m^*(t) = \min_{x \in [0,1]} \xi_m(t,x).
\end{equation}
{\begin{lemma}[Lemma 6 in \cite{G07}]\label{rev_lm}
Let $\omega$ be any non-empty open subset of $(0, 1)$. Let $ u_0 \in L^2(0,1),$ $f \in L^2(0,T; L^2(0,1))$ and $f_1, f_2 \in L^2(0,T)$. Then there exists a constant $ C = C(\omega) > 0 $ such that the solution $u \in L^2(0,T;H^1(0,1)) \cap L^\infty(0,T;L^2(0,1)) $ of
\begin{equation*}
	\begin{cases}
		- u_t - u_{xx} = f & \text{ in }   (0,T) \times (0,1),\\
		u_x(t,0)=f_1(t), \quad u_x(t,1)=f_2(t) & \text{ in }   (0,T)\\
	 u|_{t=T} = u_0& \text{ in } (0,1)
	\end{cases}
	\end{equation*}
satisfies
\begin{align}\label{rev_est1}
\nonumber&\iintQ \bigg( s \lambda^2 e^{-2s \vphi_m} \xi_m |u_x|^2 
+ s^3 \lambda^4 e^{-2s \vphi_m} \xi_m^3 |u|^2 \bigg) 
\leq C \Bigg(
\iint_{(0,T)\times\omega} s^3 \lambda^4 e^{-2s \vphi_m} \xi_m^3 |u|^2 
\\&+ \iintQ e^{-2s \vphi_m} |f|^2 
+ \int_0^T s\lambda e^{-2s \vphi_m^*(t)} \xi_m^*(t) 
\left| f_2(t) \right|^2 
- \int_0^T s\lambda e^{-2s \vphi_m^*(t)} \xi_m^*(t) 
\left| f_1(t) \right|^2 
\Bigg)
\end{align}
for any $ \lambda \geq C $ and $ s \geq C(T^{2m} + T^{2m-1}) $.
\end{lemma}}

Using the above lemma, we obtain the Carleman estimate for the second order parabolic operator. 
\begin{theorem}\label{thm-Fur-Ima}
	Let $\vphi_m$ and $\xi_m$ be given by \eqref{weight_function}. Then, there exist positive constants $\lambda_1$, $\mu_1$ and $C$, such that
{	\begin{equation}\label{car h}
		\begin{aligned} 
			\mc{I}_H(s,\lambda; \sigma)
			&	\leq C \left(\iintQ e^{-2s \vphi_m} |\sigma_x|^2 + \iintQ e^{-2s \vphi_m} |\sigma_{xx}|^2+ s^3\lambda^4 \iint_{(0,T)\times\mathcal{O}_0}  e^{-2s\vphi_m} \xi^3_m |\sigma_{x}|^2\right)\end{aligned} 
	\end{equation} }
for every $\lambda \geq \lambda_1$, $s\geq s_1$, where $s_1:={\mu_1(e^{mCT} T^{m}+T^{2m-1}+T^{2m})}$, and 
	$\sigma\in L^2(0,T; H^3(0,1)\cap H^1_{0}(0,1))\cap H^1(0,T; H^1(0,1))$ with $L\sigma:=\sigma_t+\gamma \sigma_{xx}+\bar u(x) \sigma_x={0}$.  
Here the constants $\lambda_1,\mu_1$ are independent of $T$ and the constant $C$ is independent of $T, \lambda, s$ but may depend on $\lambda_1, \mu_1, \mathcal{O}_0$ and the coefficients in $L$.
\end{theorem}
{ This result is similar to \cite[Theorem 3.1]{Cerpa-Mercado-Pazoto}. In comparison to the mentioned reference, here we consider the variable coefficient of $\sigma_x$ in the definition of $L$ and we determine the dependency of the parameter $s_1$ on $T$ explicitly. We provide a sketch of the proof of the above theorem which is similar to that of \cite[Theorem 3.1]{Cerpa-Mercado-Pazoto}.  
\paragraph{Proof of \Cref{thm-Fur-Ima}}
{ Let $\sigma \in L^2(0,T; H^3(0,1)\cap H^1_{0}(0,1))\cap H^1(0,T; H^1(0,1))$ satisfying $L\sigma=0$, where $L$ is defined in \Cref{thm-Fur-Ima}. Thus, $\sigma_x\in L^2(0, T; H^2(0, 1)) \cap L^\infty(0, T; L^2(0, 1))$ and $\sigma_x$ satisfies $(\sigma_x)_t+\gamma (\sigma_x)_{xx}=-\bar u'\sigma_x-\bar u \sigma_{xx}.$} 
Thus using \Cref{rev_lm} for $\sigma_x,$ we have
\begin{align}\label{rev_car_int}
\notag&\iintQ s \lambda^2 e^{-2s  \vphi_m} \xi_m |\sigma_{xx}|^2 + s^3 \lambda^4 \iintQ e^{-2s  \vphi_m} \xi_m^3 |\sigma_x|^2
\leq C \bigg( \iintQ e^{-2s  \vphi_m} |\sigma_x|^2 \\
&+ \iintQ e^{-2s  \vphi_m} |\sigma_{xx}|^2+
\iint_{(0,T)\times\mathcal{O}_0}  s^3 \lambda^4 e^{-2s  \vphi_m} \xi_m^3 |\sigma_x|^2 + 
\int_0^T s \lambda e^{-2s  \vphi_m^*(t)} \xi_m^*(t) |\sigma_{xx}(t, 1)|^2
\bigg),
\end{align}
where $s \geq C(T^{2m} + T^{2m-1})$, $C(\norm{\bar u'}_{L^{\infty}(0,1)}, \norm{\bar u}_{L^{\infty}(0,1)}, \mathcal{O}_0)>0$ independent of $T, \lambda, s$. 
{ We will find an estimate for the boundary term on the right-hand side of \eqref{rev_car_int}.
Note that for $\sigma \in L^2(0, T; H^3(0, 1)\cap H^1_0(0, 1))\cap H^1(0, T; H^1(0, 1))$, using the embedding $H^{\kappa}(0,1)\hookrightarrow C([0, 1]) \text{ for } \kappa>\frac{1}{2}$, we have for a.e. 
$t\in (0, T)$, $|\sigma_{xx}(t, 1)| \le C \|\sigma_{xx}(t,\cdot)\|_{H^\kappa(0, 1)}$ for any $\frac{1}{2}<\kappa<1$, where $C$ is the embedding constant. 
Thus we get 
\begin{align}\label{est-carle1}
s\lambda \int_0^T e^{-2s  \vphi_m^*(t)} \xi_m^*(t) |\sigma_{xx}(t, 1)|^2 &\le C s\lambda \int_0^T e^{-2s\vphi_m^*(t)} \xi_m^*(t) \|\sigma_{xx}(t, \cdot)\|^2_{H^{\kappa}(0, 1)} \nonumber \\
&\le C s\lambda \int_0^T e^{-2s\vphi_m^*(t)} \xi_m^*(t) \|\sigma(t, \cdot)\|^2_{H^{\kappa+2}(0, 1)}, &\quad \mbox{for any}\quad \frac{1}{2}<\kappa<1.
\end{align}
Denoting $r=\kappa+2$ with $\frac{1}{2}<\kappa<1$, we get $r \in \left( \frac{5}{2}, 3 \right)$. 
Let us take $m > 3$, and set $r=\frac{3m+1}{m+1}$ noting that $\frac{3m+1}{m+1} \in \left( \frac{5}{2}, 3 \right).$
Next we denote $\theta=\frac{1}{m+1} \in \left(0, \frac{1}{4} \right),$ which further implies $r = \theta + 3(1 - \theta)$. Using interpolation theorem
for Sobolev spaces, from \eqref{est-carle1}, we have
\begin{equation}\label{rev_est2}
\begin{array}{l}
\displaystyle
 s \lambda \int_0^T  e^{-2s  \vphi_m^*(t)} \xi_m^*(t) |\sigma_{xx}(t, 1)|^2\leq C s\lambda \int_0^T e^{-2s\vphi_m^*(t)} \xi_m^*(t) \|\sigma(t, \cdot)\|_{H^r(0,1)}^2,   \\
\displaystyle \hspace{1cm} \leq C s\lambda \int_0^T e^{-2s\vphi_m^*(t)} \xi_m^*(t) 
\left( \|\sigma(t, \cdot)\|_{H^1(0,1)}^\theta \|\sigma(t, \cdot)\|_{H^3(0,1)}^{1 - \theta} \right)^2\\
\displaystyle \hspace{1cm} = C\lambda \int_0^T (s e^{-2s\vphi_m^*(t)}\xi_m^*(t))^{3\theta} \|\sigma(t, \cdot)\|_{H^1(0,1)}^{2\theta} 
	(se^{-2s\vphi_m^*(t)}\xi_m^*(t))^{1 - 3\theta} \|\sigma(t, \cdot)\|_{H^3(0,1)}^{2(1 - \theta)} \notag \\
 \displaystyle \hspace{1cm} \leq \frac{1}{4} s^3 \lambda \int_0^T e^{-2s\vphi_m^*(t)} (\xi_m^*(t))^3 \|\sigma(t, \cdot)\|_{H^1(0,1)}^2 
	+ C s^p \lambda \int_0^T e^{-2s\vphi_m^*(t)} (\xi_m^*(t))^p \|\sigma(t, \cdot)\|_{H^3(0,1)}^2,  \notag 
\end{array}
\end{equation}
where 
$p = \frac{1 - 3\theta}{1 - \theta} = \frac{3r - 7}{r - 1}=\left(1-\frac{2}{m}\right),
$
and $C>0$ a constant independent of $T.$
}


Let us define 
$\eta(t) = s^{\frac{1}{2}-
\frac{1}{m}}  \sqrt{\lambda} e^{-s \vphi^*_m} (\xi^*_m)^{\frac{1}{2} - \frac{1}{m}}
$ in $[0,T]$
and 
$
\sigma^* = \eta \sigma.
$
As $L\sigma=0,$ $\sigma^*$ satisfies the following equation
\begin{equation}\label{rev_parabolic}
	\begin{cases}
\sigma^*_t+\gamma \sigma^*_{xx}+\bar u \sigma^*_x  =  \eta' \sigma &\text{ in }   (0,T) \times (0,1),\\
\sigma^*(t,0)=\sigma^*(t,1)=0 & \text{ in }   (0,T),\\
\sigma^*(T,x)=0 & \text{ in }   (0,1).
\end{cases}
\end{equation}
Using the definition of $\eta$ and $\sigma^*$, one can directly obtain the following identity
\begin{equation} \label{rev_est3}
	s^p \lambda \int_0^T e^{-2 s \vphi^*_m} (\xi^*_m)^p \|\sigma\|^2_{H^3(0,1)} 
	=  \|\sigma^*\|^2_{L^2(0,T; H^3(0,1))}.
\end{equation}
By parabolic regularity for the equation \eqref{rev_parabolic} and Poincar\'e inequality for $\sigma$, we get
\begin{align}\label{rev_est4}
\|\sigma^*\|^2_{L^2(0,T; H^3(0,1))} 
\leq Ce^{CT} \|\eta' \sigma\|^2_{L^2(0,T; H^1_0(0,1))} 
\leq Ce^{CT} \|\eta' \sigma_x\|^2_{L^2(0,T; L^2(0,1))}, 
\end{align}
where $C$ is a positive constant independent of $T,$ see \Cref{rev_app} for details. 

Let us first estimate the term in right hand side of the above inequality. Thanks to \eqref{weight-deri-t}, we have
\begin{align*}
	\|\eta' \sigma_x\|_{L^2(0,T; L^2(0,1))}\leq C s^{\frac{1}{2} - \frac{1}{m}}  \sqrt{\lambda} \left(s T+ T^{2m+1}\right)\| e^{-s \vphi^*_m} (\xi^*_m)^{\frac{3}{2} }\sigma_x\|_{L^2(0,T; L^2(0,1))}.
\end{align*}
As, $s\geq C_1 T^{2m}$ for some $C_1>0$ independent of $T$, we can further write the above inequality as
\begin{align*}
	\|\eta' \sigma_x\|_{L^2(0,T; L^2(0,1))}\leq C  T s^{\frac{3}{2} - \frac{1}{m}}  \sqrt{\lambda} \| e^{-s \vphi^*_m} (\xi^*_m)^{\frac{3}{2}}\sigma_x\|_{L^2(0,T; L^2(0,1))},
\end{align*}
where $C>0$ is a positive constant independent of $T.$  
Using the above estimate in \eqref{rev_est4}, we get
\begin{align}\label{rev_est6}
	\|\sigma^*\|^2_{L^2(0,T; H^3(0,1))} 
	\leq Ce^{CT} T^2 s^{{3} - \frac{2}{m}}  {\lambda} \| e^{-s \vphi^*_m} (\xi^*_m)^{\frac{3}{2}}\sigma_x\|^2_{L^2(0,T; L^2(0,1))}.
\end{align}
Further assuming $s\geq e^{mCT}T^m,$ we obtain the following estimate
\begin{align*}
		\|\sigma^*\|^2_{L^2(0,T; H^3(0,1))} \leq C   s^{{3}}  {\lambda} \| e^{-s \vphi^*_m} (\xi^*_m)^{\frac{3}{2}}\sigma_x\|^2_{L^2(0,T; L^2(0,1))}.
\end{align*}
Thus, combining \eqref{rev_est2}, \eqref{rev_est3} and the above estimate we deduce
\begin{align}\label{rev_est7}
\notag&	s\lambda \int\int_{\Sigma} e^{-2s\vphi_m^*} \xi_m^* |\sigma_{xx}|^2 \\
	&\leq \frac{C_2}{4} s^3 \lambda \int_0^T e^{-2s\vphi_m^*} (\xi_m^*)^3 |\sigma_x|^2+ C_3  s^{{3}}  {\lambda} \| e^{-s \vphi^*_m} (\xi^*_m)^{\frac{3}{2}}\sigma_x\|^2_{L^2(0,T; L^2(0,1))},
	\end{align}
for some $ C_2, C_3>0$ independent of $T.$
Therefore choosing ${\lambda^3}>\max\{\frac{C_2}{4}, C_3\}>0$ and using  \eqref{rev_wght}, one can write the following
\begin{align}\label{rev_est5}
		s \lambda \int_0^T e^{-2s  \vphi_m^*(t)} \xi_m^*(t) |\sigma_{xx}(t, 1)|^2 \leq   \frac{1}{2} s^3  {\lambda}^4\iintQ   e^{-2s \vphi_m} \xi_m^{{3}}|\sigma_x|^2,
\end{align}
where $s\geq C(e^{mCT} T^{m}+T^{2m-1}+T^{2m})$, for some constant $C>0$ independent of $T, s, \lambda$. Note that
\eqref{rev_est5} can be absorbed by the left hand side of \eqref{rev_car_int}, and we get the desired result.
}

Next, we mention the following Carleman estimates for the fourth order parabolic operator. 
\begin{theorem}\label{carl ch}
	Let $\vphi_m$ and $\xi_m$ be given by \eqref{weight_function}.   Then, there exist positive constants $\lambda_2$, $\mu_2$ and $C$, such that
	\begin{equation}\label{car ch}
		\begin{aligned} 
			\mc{I}_{CH}(s,\lambda; v)
			&	\leq  s^7\lambda^8 \iint_{(0,T)\times\mathcal{O}_0}  e^{-2s\vphi_m} \xi^7_m |v|^2+C\iintQ  e^{-2s\vphi_m}  |Pv|^2, \end{aligned} \end{equation}
	for every $\lambda \geq \lambda_2$, $s\geq s_2$, where $s_2:=\mu_2(T^{2m-1}+T^{2m}),$ and 
	$v\in L^2(0,T; H^4(0,1))\cap H^1(0,T; L^2(0,1))$ with $Pv=v_t-v_{xxxx}-\gamma_2 v_{xx} + \bar u(x) v_x+\bar u'(x) v$.  
Here the constants $\lambda_2, \mu_2$ are independent of $T$ and the constant $C$ is independent of $T, \lambda, s$ but may depend on $\lambda_2, \mu_2, \mathcal{O}_0$ and the coefficients in $P$.
\end{theorem}

This theorem was proved in \cite[Proposition 3.1]{Guzman},  for the weight defined in \eqref{weight_function} with $ m= 1 $.   In order to prove \Cref{carl ch},  for the  weight function defined   in \eqref{weight_function} with $m>3$  one can follow the same steps as in  the proof  of \cite[Proposition 3.1]{Guzman}. Note that  the only changes occur in the time derivative estimates of weight functions and by taking into account the estimate \eqref{weight_estimate-t-m} we  get the inequality \eqref{car ch}.
 We also refer \cite{Zhou}, \cite{Cerpa-Mercado-Pazoto} where a similar Carleman estimate has been addressed with different boundary conditions.

Our next task is to combine the above two Carleman estimates and derive a joint Carleman estimate for the adjoint equation \eqref{CH adj} which essentially helps to derive the observability inequality \eqref{obser-1}. 
\begin{theorem}[\textbf{Joint Carleman estimate for} \eqref{CH adj}]\label{Thm_Carleman_main}
Let the weight functions $(\vphi_m, \xi_m)$  be given by \eqref{weight_function}. Then, there exist positive constants $\lambda_0$, $\mu_0$ and $C$ such that for any $(\sigma_T, v_T)\in (L^2(0,1))^2$, $(\sigma, v)$, the solution to \eqref{CH adj}, satisfies
			\begin{align}\label{carleman-joint}
				s^3\lambda^4 \iintQ e^{-2s\vphi_m} \xi^3_m |\sigma_x|^2+ 
			s^7\lambda^8 \iintQ e^{-2s\vphi_m} \xi_m^7 |v|^2\leq C s^{39}\lambda^{40} \iint_{(0,T)\times\mathcal{O}} e^{-2s\vphi_m} \xi_m^{39} |v|^2.
		\end{align}
		for all $\lambda \geq \lambda_0$ and $s\geq s_0$, where $s_0:= {\mu_0(e^{mCT}T^{m}+T^{2m-1}+T^{2m})}$. 
Here the constants $\lambda_0, \mu_0$ are independent of $T$ and the constant $C$ is independent of $T$ but may depend on $\lambda_0, \mu_0$, the coefficients in \eqref{CH adj} and the set $\mathcal{O}$.
	\end{theorem}
\begin{proof}
Let us first assume that $\sigma_T, v_T\in C^{\infty}_c(0,1).$ Let $(\sigma, v)$ be the solution of the adjoint problem \eqref{CH adj} and we have $(\sigma,v)\in  L^2(0,T; H^3(0,1)\cap H^1_{0}(0,1))\cap H^1(0,T; H^1(0,1))\times L^2(0,T; H^4(0,1))\cap H^1(0,T; L^2(0,1)).$ Then, $L\sigma=0$ and $Pv=\gamma_1 \sigma_x$, where $L$ and $P$ are given in Theorems \ref{thm-Fur-Ima} and \ref{carl ch} respectively.  We can apply two Carleman estimates \eqref{car h} and \eqref{car ch} for the solution $(\sigma,v)$ of the adjoint system respectively. 
{Furthermore, \eqref{car h} holds for all $s\geq s_1=\mu_1(T^{2m}+T^{2m-1}+e^{mCT}T^{m})$, whereas Carleman estimate \eqref{car ch} for the fourth order operator holds for all $s\geq s_2=\mu_2(T^{2m}+T^{2m-1})$. It is easy to check that, if we take $\widehat{\mu}:=\max\{\mu_1, \mu_2\}$, then \
	\begin{equation*}\widehat{\mu}\left(T^{2m}+T^{2m-1}+e^{mCT}T^{m}\right)\ge \max\{\mu_2(T^{2m}+T^{2m-1}), \mu_1(T^{2m}+T^{2m-1}+e^{mCT} T^{m})\}.
 \end{equation*}
	Therefore, both Carleman estimates hold true when $s\geq \widehat{s}:=\widehat{\mu}(T^{2m}+T^{2m-1}+e^{mCT} T^{m}).$}

Adding the Carleman estimates \eqref{car h} and \eqref{car ch}, we have 
	\begin{multline}\label{Add_carlemans}
		\mc{I}_{CH}(s,\lambda; v) + \mc{I}_{H}(s,\lambda; \sigma) 
	\leq C \bigg[{\iintQ  e^{-2s\vphi_m}|\sigma_{xx}|^2}+\iintQ  e^{-2s\vphi_m}|\sigma_x|^2 \\+
	 s^7\lambda^8 \int_0^T\int_{\mathcal{O}_0} e^{-2s\vphi_m} \xi_m^7 |v|^2 
	 + s^3\lambda^4 \int_0^T\int_{\mathcal{O}_0} e^{-2s\vphi_m} \xi_m^3 |\sigma_x|^2\bigg],
	\end{multline}
for all $\lambda\geq \widehat{\lambda}:= \max\{\lambda_1, \lambda_2\}$ and $s\geq \widehat{s}= \widehat\mu{(e^{mCT}T^{m}+T^{2m-1}+T^{2m})}$
 for some positive constant $C$ as obtained in Theorems \ref{thm-Fur-Ima} and \ref{carl ch}. 
\smallskip
 
\noindent 
{{\bf Step 1. Absorbing the lower order integrals.} Observing $1\leq  T^{2m}\xi_m(t,x), \forall (t,x)\in Q_T$, we get 
\begin{align}\label{lower_int}
\iintQ e^{-2s\vphi_m}|\sigma_{xx}|^2 \leq  T^{2m}\iintQ e^{-2s\vphi_m}\xi_m|\sigma_{xx}|^2, \\ \nonumber
\iintQ e^{-2s\vphi_m}|\sigma_x|^2 \leq  T^{6m}\iintQ e^{-2s\vphi_m}\xi_m^3|\sigma_x|^2.
\end{align}
For all $s\geq \max\{\widehat{s}, (C^{1/3}+C) T^{2m}\}$ and $\lambda\geq \max\{\widehat{\lambda}, \sqrt{2}\}$, we get 
\begin{align*}
C T^{2m}\iintQ e^{-2s\vphi_m}\xi_m|\sigma_{xx}|^2 \leq  \frac{1}{2} s\lambda^2 \iintQ e^{-2s\vphi_m}\xi_m|\sigma_{xx}|^2, \\
	C T^{6m}\iintQ e^{-2s\vphi_m}\xi_m^3|\sigma_x|^2 \leq  \frac{1}{2} s^3\lambda^4 \iintQ e^{-2s\vphi_m}\xi_m^3|\sigma_x|^2.
\end{align*}
Thus, the first and second terms in the right hand side of \eqref{Add_carlemans} can be absorbed respectively by the first and second integrals of $\mc{I}_{H}(s,\lambda; \sigma)$ appearing in the left hand side of \eqref{Add_carlemans}.} Finally it leads to the following 
\begin{align}\label{Add_carlemans-2}
	\mc{I}_{CH}(s,\lambda; v) + \mc{I}_{H}(s,\lambda; \sigma) 
	\leq C \bigg[
	s^7\lambda^8 \int_0^T\int_{\mathcal{O}_0} e^{-2s\vphi_m} \xi_m^7 |v|^2 
	+ s^3\lambda^4 \int_0^T\int_{\mathcal{O}_0} e^{-2s\vphi_m} \xi_m^3 |\sigma_x|^2\bigg],
\end{align}
for all $\lambda\geq \lambda_0$ and $s\geq \mu_0{(e^{mCT}T^{m}+T^{2m-1}+T^{2m})}$ where $\lambda_0$ and $\mu_0$ can be chosen suitably large enough satisfying $\lambda_0\geq \widehat{\lambda}$ and $\mu_0\geq \widehat{\mu}$. 

\smallskip 
\noindent 
 {\bf Step 2. Absorbing the observation integral in $\sigma_x$.}
Consider $\mathcal{O}_0$ and another non-empty open set $\mathcal{O}_1$ in such a way that  $\mathcal{O}_0 \subset \subset \mathcal{O}_1 \subset \subset \mathcal{O}$. Then, we consider the function 
\begin{align}\label{smooth_func}
	\zeta \in C^\infty_c(\mathcal{O}_1)  \ \text{ with } \ 0\leq \zeta \leq 1, \ \text{ and } \ \zeta =1 \ \text{ in } \ \mathcal{O}_0. 
\end{align}
Now, recall the adjoint system \eqref{CH adj},  one has (since $\gamma_1\neq 0$),
\begin{align*}
	\sigma_x = \frac{1}{\gamma_1} \left(v_t-v_{xxxx}-\gamma_2 v_{xx} + \bar u(x) v_x+\bar u'(x) v   \right), \quad \text{in } Q_T.
\end{align*}
Using this, we have 
\begin{align}\label{Obs-psi}
s^3\lambda^4 \int_0^T\int_{\mathcal{O}_0}& e^{-2s\vphi_m} \xi_m^3 |\sigma_x|^2 
\leq s^3\lambda^4 \int_0^T\int_{\mathcal{O}_1} \zeta e^{-2s\vphi_m} \xi_m^3 |\sigma_x|^2   \notag \\ 
&= \frac{1}{\gamma_1}s^3\lambda^4 \int_0^T\int_{\mathcal{O}_1} \zeta e^{-2s\vphi_m} \xi_m^3 \sigma_x \left(v_t-v_{xxxx}-\gamma_2 v_{xx} + \bar u(x) v_x+\bar u'(x) v   \right)  \notag \\ 
 &= \frac{1}{\gamma_1}s^3\lambda^4 \int_0^T\int_{\mathcal{O}_1} \zeta e^{-2s\vphi_m} \xi_m^3 \sigma_x \left((v_t-\gamma_1 v_{xx}) -v_{xxxx}+(\gamma_1-\gamma_2) v_{xx} + \bar u(x) v_x+\bar u'(x) v   \right)  \notag \\ 
 &=I_1+ I_2+ I_3+I_4+I_5 .
\end{align}
\noindent
(i) Let us start with the following. 
\begin{align}\label{A_1}
	I_1 &=  \frac{1}{\gamma_1}s^3\lambda^4 \int_0^T\int_{\mathcal{O}_1} \zeta e^{-2s\vphi_m} \xi_m^3   {(-\sigma_t-\gamma_1 \sigma_{xx})}_x v +R 
	\end{align}
where \begin{align*}
R=&-\frac{1}{\gamma_1} s^3\lambda^4 \int_0^T\int_{\mathcal{O}_1} \zeta (e^{-2s\vphi_m} \xi_m^3)_t  \sigma_x v
+ \frac{1}{\gamma_1}s^3\lambda^4 \int_0^T\int_{\mathcal{O}_1} (\zeta e^{-2s\vphi_m} \xi_m^3)_x  \sigma_x v_x\\
&- \frac{1}{\gamma_1}s^3\lambda^4 \int_0^T\int_{\mathcal{O}_1} (\zeta e^{-2s\vphi_m} \xi_m^3)_x  \sigma_{xx} v
\end{align*}
Again using the first equation of \eqref{CH adj}, \begin{align}\label{A_1}
	I_1 &=  \frac{1}{\gamma_1}s^3\lambda^4 \int_0^T\int_{\mathcal{O}_1} \zeta e^{-2s\vphi_m} \xi_m^3   {(\bar u(x)\sigma_x)}_x v +R \\
	\notag&\leq C\frac{1}{\gamma_1}s^3\lambda^4 \int_0^T\int_{\mathcal{O}_1} \zeta e^{-2s\vphi_m} \xi_m^3   |{\sigma_x}_x| |v| +C\frac{1}{\gamma_1}s^3\lambda^4 \int_0^T\int_{\mathcal{O}_1} \zeta e^{-2s\vphi_m} \xi_m^3   |{\sigma}_x| |v|+|R|\\
\notag	&\leq \epsilon s\lambda^2  \iintQ e^{-2s\vphi_m} \xi_m |\sigma_{xx}|^2+ C(\epsilon) s^5\lambda^6\int_0^T\int_{\mathcal{O}_1}  e^{-2s\vphi_m} \xi_m^5   |v|^2+ \epsilon s^3\lambda^4 \iintQ e^{-2s\vphi_m} \xi^3_m |\sigma_{x}|^2\\
	&+C(\epsilon) s^3\lambda^4 \iintQ e^{-2s\vphi_m} \xi^3_m |v|^2+|R|.
\end{align}
In above, we perform an integration by parts with respect to $t$. There is no boundary terms since   the weight function $e^{-2s\vphi_m}$ vanishes near $t=0$ and $T$.  

Recall from \eqref{weight_estimate-t-m}, for all $s\ge \mu_0{(e^{mCT}T^{m}+T^{2m-1}+ T^{2m})}$, we have 
$$
|\partial_t(e^{-2s\vphi_m} \xi^l_m)| \leq C s^2 \xi^2_m (e^{-2s\vphi_m} \xi^l_m),  \quad (t, x)\in Q_T,
$$
for some positive constant $C$, independent of $\lambda, s, t, T$. 

Using this estimate along with \eqref{weight-deri-x}, the Cauchy-Schwarz inequality and noting that $m>3,$ we get
\begin{align*}
	|R| &\leq \epsilon s^3\lambda^4 \iintQ e^{-2s\vphi_m} \xi^3_m |\sigma_{x}|^2+ \epsilon s\lambda^2  \iintQ e^{-2s\vphi_m} \xi_m |\sigma_{xx}|^2\\
	&+ C(\epsilon) s^7\lambda^8\int_0^T\int_{\mathcal{O}_1}  e^{-2s\vphi_m} \xi_m^7   |v|^2+ C(\epsilon) s^5\lambda^6\int_0^T\int_{\mathcal{O}_1} e^{-2s\vphi_m} \xi_m^5   |v_x|^2.
\end{align*}
Thus combining we have
\begin{align}\label{a1}
\nonumber	I_1\leq& \epsilon s\lambda^2  \iintQ e^{-2s\vphi_m} \xi_m |\sigma_{xx}|^2+  \epsilon s^3\lambda^4 \iintQ e^{-2s\vphi_m} \xi^3_m |\sigma_{x}|^2\\
	&+C(\epsilon) s^7\lambda^8 \iintQ e^{-2s\vphi_m} \xi^7_m |v|^2+C(\epsilon) s^5\lambda^6\int_0^T\int_{\mathcal{O}_1} e^{-2s\vphi_m} \xi_m^5   |v_x|^2.
\end{align}

\noindent
(ii) Next, after integrating by parts twice with respect to $x$, we get 
\begin{align}\label{A_2}
	I_2:&=  -\frac{1}{\gamma_1} s^3 \lambda^4 \int_0^T\int_{\mathcal{O}_1} \zeta e^{-2s\vphi_m} \xi_m^3 \sigma_x v_{xxxx}  \notag  \\
	&= \frac{1}{\gamma_1} s^3\lambda^4 \bigg[\int_0^T\int_{\mathcal{O}_1} \zeta_x (e^{-2s\vphi_m} \xi_m^3) \sigma_x v_{xxx} 
	+\int_0^T\int_{\mathcal{O}_1} \zeta (e^{-2s\vphi_m} \xi_m^3)_x \sigma_x v_{xxx}   \notag  \\
	& \qquad \qquad \qquad \qquad \qquad \qquad \qquad +\int_0^T\int_{\mathcal{O}_1} \zeta (e^{-2s\vphi_m} \xi_m^3) \sigma_{xx} v_{xxx}\bigg].  \end{align}
Using \eqref{weight-deri-x} and Cauchy-Schwarz inequality,    we have from \eqref{A_2}, for any $\epsilon>0$
\begin{align}\label{A_2-1}
	|I_2| \leq 2\epsilon s^3\lambda^4 \iintQ e^{-2s\vphi_m} \xi_m^{3} |\sigma_{x}|^2+\epsilon s\lambda^2 \iintQ e^{-2s\vphi_m} \xi_m |\sigma_{xx}|^2 
+	 C({\epsilon}) s^{5}\lambda^{6} \int_0^T \int_{\mathcal{O}_1}  e^{-2s\vphi_m} \xi_m^{5} |v_{xxx}|^2 .
\end{align}

\noindent
(iii) Let us estimate the third term of \eqref{Obs-psi} in the following way 
\begin{align*}
	I_3:&=  \frac{\gamma_1-\gamma_2}{\gamma_1}s^3\lambda^4\int_0^T\int_{\mathcal{O}_1} \zeta e^{-2s\vphi_m} \xi_m^3 \sigma_x v_{xx}
\end{align*}
Therefore, for $\epsilon>0$ we have, using \eqref{weight-deri-x} and the Cauchy-Schwarz inequality, that 
\begin{align}\label{A_3} 
	|I_3| \leq \epsilon s^3\lambda^4 \iintQ e^{-2s\vphi_m} \xi_m^3 |\sigma_x|^2 + 
	C({\epsilon}) s^3\lambda^4 \int_0^T \int_{\mathcal{O}_1} e^{-2s\vphi_m} \xi_m^3 |v_{xx}|^2.  
\end{align}

\noindent
(iv) Next, it is easy to see that 
\begin{align}\label{A_4}
	|I_4|+|I_5|  \leq \epsilon s^3\lambda^4 \iintQ e^{-2s\vphi_m} \xi_m^3 |\sigma_x|^2 + 
	C({\epsilon}) s^3\lambda^4 \int_0^T \int_{\mathcal{O}_1} e^{-2s\vphi_m} \xi_m^3 |v_{x}|^2\nonumber\\+ 
	C({\epsilon}) s^3\lambda^4 \int_0^T \int_{\mathcal{O}_1} e^{-2s\vphi_m} \xi_m^3 |v|^2.
\end{align}
Combining the estimates \eqref{Obs-psi}, \eqref{a1}, \eqref{A_2-1}, \eqref{A_3}, \eqref{A_4}, and applying in \eqref{Obs-psi} we get 
\begin{align}\label{estimate-aux-1}
s^3\lambda^4 \int_0^T\int_{\mathcal{O}_0} e^{-2s\vphi_m} \xi_m^3 |\sigma_x|^2 
\nonumber	&\leq
5\epsilon s^3\lambda^4 \iintQ e^{-2s\vphi_m} \xi_m^3 |\sigma_x|^2 + \epsilon s\lambda^2 \iintQ e^{-2s\vphi_m} \xi_m |\sigma_{xx}|^2  \\
\nonumber		&+C(\epsilon) s^7\lambda^8 \int_0^T\int_{\mathcal{O}_1} e^{-2s\vphi_m} \xi^7_m |v|^2+C(\epsilon) s^5\lambda^6\int_0^T\int_{\mathcal{O}_1} e^{-2s\vphi_m} \xi_m^5   |v_x|^2\\ 
		&+C({\epsilon}) s^3\lambda^4 \int_0^T \int_{\mathcal{O}_1} e^{-2s\vphi_m} \xi_m^3 |v_{xx}|^2+	 C({\epsilon}) s^{5}\lambda^{6} \int_0^T \int_{\mathcal{O}_1}  e^{-2s\vphi_m} \xi_m^{5} |v_{xxx}|^2 .
\end{align} 
Now, fix $\epsilon>0$  small enough in \eqref{estimate-aux-1}, so that all the terms with coefficient $\epsilon$ can be absorbed by the associated integrals in the left hand side of \eqref{Add_carlemans-2}. As a  consequence,  the estimate \eqref{Add_carlemans-2} boils  down to 
\begin{align}\label{Add_carlemans-3 1}
\notag	\mc{I}_{CH}(s,\lambda; v) + \mc{I}_{H}(s,\lambda; \sigma) 
	&\leq C(\epsilon) s^7\lambda^8 \int_0^T\int_{\mathcal{O}_1} e^{-2s\vphi_m} \xi^7_m |v|^2+C(\epsilon) s^5\lambda^6\int_0^T\int_{\mathcal{O}_1} e^{-2s\vphi_m} \xi_m^5   |v_x|^2\\ 
	&+C({\epsilon}) s^3\lambda^4 \int_0^T \int_{\mathcal{O}_1} e^{-2s\vphi_m} \xi_m^3 |v_{xx}|^2+	 C({\epsilon}) s^{5}\lambda^{6} \int_0^T \int_{\mathcal{O}_1}  e^{-2s\vphi_m} \xi_m^{5} |v_{xxx}|^2 ,
\end{align}
for all $\lambda\geq \lambda_0$ and $s\geq \mu_0{(e^{mCT}T^{m}+T^{2m-1}+ T^{2m})}$.

\smallskip 
\noindent 
{\bf Step 3. Absorbing the observation integral in $v_x, v_{xx}, v_{xxx}$.}
We need to estimate the last integral of \eqref{Add_carlemans-3 1}. Consider a function (recall that $\mathcal{O}_1\subset\subset \mathcal{O}_2\subset \subset \mathcal{O}$)
\begin{align}\label{smooth_func-2}
	\wphi \in C^\infty_c(\mathcal{O}_2)  \ \text{ with } \ 0\leq \wphi \leq 1, \ \text{ and } \  \wphi =1 \ \text{ in } \ \mathcal{O}_1. 
\end{align}
With this, we have 
\begin{align*}  
	s^{5}\lambda^{6} \int_0^T \int_{\mathcal{O}_1}  e^{-2s\vphi_m} \xi_m^{5} |v_{xxx}|^2 &\leq s^{5}\lambda^{6} \int_0^T \int_{\mathcal{O}_2} \wphi e^{-2s\vphi_m} \xi_m^{5} |v_{xxx}|^2 \\
	&=-s^{5}\lambda^{6} \int_0^T \int_{\mathcal{O}_2} \wphi e^{-2s\vphi_m} \xi_m^{5} v_{xx}v_{xxxx}+\frac{s^{5}\lambda^{6}}{2} \int_0^T \int_{\mathcal{O}_2} {\left(\wphi e^{-2s\vphi_m} \xi_m^{5}\right)}_{xx} v_{xx}^2\\
	&\leq \epsilon s^{-1} \iintQ e^{-2s\vphi_m} \xi_m^{-1} |v_{xxxx}|^2 + C(\epsilon) s^{11}\lambda^{12} \int_{\mathcal{O}_2} { e^{-2s\vphi_m} \xi_m^{11}} v_{xx}^2
\end{align*}

Consider a function (recall that $\mathcal{O}_2\subset\subset \mathcal{O}_3\subset \subset \mathcal{O}$)
\begin{align}\label{smooth_func-21}
	\wphi_1 \in \C^\infty_c(\mathcal{O}_3)  \ \text{ with } \ 0\leq \wphi_1 \leq 1, \ \text{ and } \  \wphi_1 =1 \ \text{ in } \ \mathcal{O}_2. 
\end{align}
With this, we have 
\begin{align*}  
	s^{11}\lambda^{12} \int_0^T \int_{\mathcal{O}_2}  e^{-2s\vphi_m} \xi_m^{11} |v_{xx}|^2 &\leq s^{11}\lambda^{12} \int_0^T \int_{\mathcal{O}_3} \wphi_1 e^{-2s\vphi_m} \xi_m^{11} |v_{xx}|^2 \\
	&=-s^{11}\lambda^{12} \int_0^T \int_{\mathcal{O}_2} \wphi e^{-2s\vphi_m} \xi_m^{11} v_{x}v_{xxx}+\frac{s^{11}
\lambda^{12}}{2} \int_0^T \int_{\mathcal{O}_2} {\left(\wphi e^{-2s\vphi_m} \xi_m^{11}\right)}_{xx} v_{x}^2\\
	&\leq \epsilon s\lambda^2 \iintQ e^{-2s\vphi_m} \xi_m |v_{xxx}|^2 + C(\epsilon) s^{21}\lambda^{22} \int_0^T\int_{\mathcal{O}_2} { e^{-2s\vphi_m} \xi_m^{21}} |v_{x}|^2.
\end{align*}
Again consider a function (recall that $\mathcal{O}_3\subset\subset \mathcal{O}_4\subset \subset \mathcal{O}$)
\begin{align}\label{smooth_func-213}
	\wphi_2 \in C^\infty_c(\mathcal{O}_4)  \ \text{ with } \ 0\leq \wphi_1 \leq 1, \ \text{ and } \  \wphi_1 =1 \ \text{ in } \ \mathcal{O}_3. 
\end{align}

Again by using the information \eqref{weight-deri-x} and Cauchy-Schwarz inequality,   we get for some $\epsilon>0$, that 
\begin{multline} \label{A_2-2}
	s^{21}\lambda^{22} \int_{\mathcal{O}_2} { e^{-2s\vphi_m} \xi_m^{21}} v_{x}^2    
	\leq \epsilon s^3\lambda^4 \iintQ e^{-2s\vphi_m} \xi^3_m |v_{xx}|^2 + C(\epsilon) s^{39}\lambda^{40}\int_0^T \int_{\mathcal{O}} { e^{-2s\vphi_m} \xi_m^{39}} v^2.
\end{multline}
We fix small enough $\epsilon>0$ in \eqref{A_2-2} so that the integrals with coefficient $\epsilon$ can be absorbed by the leading integrals in the left hand of \eqref{Add_carlemans-3 1} and as a result we obtain 
\begin{align*}
	\mc{I}_{CH}(s,\lambda; v) + \mc{I}_{H}(s,\lambda; \sigma) 
	\leq C 
	 s^{39}\lambda^{40} \int_0^T\int_{\mathcal{O}} { e^{-2s\vphi_m} \xi_m^{39}} v^2
\end{align*}
for all $\lambda\geq \lambda_0$ and $s\geq \mu_0{(e^{mCT}T^{m}+T^{2m-1}+ T^{2m})}$.

This is the required joint Carleman estimate \eqref{carleman-joint} of our theorem. The proof is finished by using standard density argument of $C^{\infty}_{c}(0,1)$ in $L^2(0,1)$. 
\end{proof}

\section{Controllability of the linear system}\label{sec-car1}
{
This section is devoted to the proof of null controllability of the linearized model \eqref{CH}--\eqref{bd}--\eqref{in}, that is \Cref{Thm-linear}. We know that the null controllability of \eqref{CH}--\eqref{bd}--\eqref{in} is equivalent to an observability inequality for the solutions of the adjoint system \eqref{CH adj}. In the sequel, we first prove the observability inequality using Carleman inequality obtained in Section \ref{Sec-carleman-1}. The observability inequality indeed gives the null controllability of the system \eqref{CH}--\eqref{bd}--\eqref{in} at any time $T>0$ with a certain control cost. To obtain the specific control cost as mentioned in \Cref{Thm-linear}, next we split the proof of the controllability result in to two cases- one case is for $0<T<1$ and the other case is for $T\ge 1$ and finally conclude the section with the proof of \Cref{Thm-linear}.  
}

\subsection{Observability inequality}\label{sec-obs}

To prove the observability inequality, we require the following lemma which can be obtained by using energy estimate.

 \begin{lemma}\label{energy est1}
		Let $(\sigma_T, v_T)\in (L^2(0,1))^2.$ Then $(\sigma, v)$, the solution of \eqref{CH adj} satisfies 
		\begin{equation}\label{aux}
			\|\sigma(0,\cdot)\|^2_{L^2(0,1)} + 	\|v(0,\cdot)\|^2_{L^2(0,1)}\leq
			\frac{C}{T} e^{CT}  \int_{\frac{T}{4}}^{\frac{3T}{4}}\int_{0}^{1} (|\sigma(t,x)|^2+|v(t,x)|^2) dx dt,
		\end{equation}
		for some positive constant $C$, independent of $T$. 
	\end{lemma}
	
\begin{proof}

First, let us consider $(\sigma_T, v_T)\in H^1_0(0, 1)\times H^2_1$. Then solution of \eqref{CH adj} lies in the space $C([0,T], H^1_0(0, 1)\times H^2_1)\cap
\left(L^2(0, T; H^2(0,1)\cap H^1_0(0,1))\times L^2(0,T;H^4(0,1)\cap H^2_1(0,1))\right).$ Multiplying the equations \eqref{CH adj} by $\sigma$ and $v$ respectively and adding them using integration by parts, we have 
	\begin{align*}
-\frac{1}{2}\frac{d}{dt}\int_{0}^{1}(\sigma^2(t,x)+v^2(t,x)) dx =&-\gamma\int_{0}^{1}\sigma_x^2+\int_{0}^{1}\bar{u}(x)\sigma_x\sigma-\int_{0}^{1} v_{xx}^2-\gamma_2\int_{0}^{1}v_{xx}v\\
&+\int_{0}^{1}\bar{u}(x)v_x v + \int_{0}^{1}\bar{u}'(x)v^2-\gamma_1\int_{0}^{1}\sigma_x v.
	\end{align*} 
Now, using Cauchy-Schwarz inequality and then ignoring all the negative term we finally obtain:
\begin{align*}
	-\frac{1}{2}\frac{d}{dt}\int_{0}^{1}\left(|\sigma(t,x)|^2+ |v(t,x)|^2\right) dx \leq C\left(\int_{0}^{1}|\sigma(t,x)|^2 \, dx+\int_{0}^{1} |v(t,x)|^2\, dx\right), \quad \forall\, t\in (0, T),
\end{align*}
for some positive constant $C$, independent of $t$. 
This inequality can be written as
\begin{align*}
	-\frac{d}{dt}\bigg[e^{2Ct}\int_{0}^{1}\left(|\sigma(t,x)|^2 + |v(t,x)|^2\right) dx\bigg] \leq 0, \quad \forall\, t\in (0, T).
\end{align*}
Integrating both sides over $[0,t]$ we have
 \begin{align*}
 	\int_{0}^{1}\left(|\sigma(0,x)|^2+ |v(0,x)|^2\right)\, dx\leq e^{2Ct}\left(\int_{0}^{1}|\sigma(t,x)|^2\, dx+ \int_{0}^{1} |v(t,x)|^2\, dx\right) .
 \end{align*}
Integrating with respect to $t$ over $\left(\frac{T}{4},\frac{3T}{4}\right)$ we get a positive constant $C>0$, independent of $T$, such that:
\begin{align*}
	\int_{0}^{1}|\sigma(0,x)|^2+ |v(0,x)|^2 \, dx \leq \frac{C}{T}e^{CT}\int_{\frac{T}{4}}^{\frac{3T}{4}}\int_{0}^{1}\left(|\sigma(t,x)|^2+ |v(t,x)|^2 \right)\, dx\, dt.
\end{align*}
Using density argument, we get the above inequality when $(\sigma_T, v_T)\in (L^2(0,1))^2.$

\end{proof}

By adapting classical arguments (see, for instance, \cite[Theorem 2.44, Chapter 2]{Co07}), it can be shown that the following observability inequality \eqref{obser-1} is equivalent to the null controllability of \eqref{CH}--\eqref{bd}--\eqref{in}.

\begin{proposition}[\textbf{Observability inequality}]\label{pr ob}
	Let $T>0$ any given time and $ m\in \mathbb{N}$ with $m>3$. There exists a constant $M>0$ depending on $\mathcal{O}$, $\gamma$, $\gamma_1$, $\gamma_2,$ but independent of $T$ such that for all $(\sigma_T, v_T)\in (L^2(0,1))^2$, $(\sigma, v)$, the solution to the adjoint system \eqref{CH adj}, satisfies
	the observability inequality:
	\begin{align}\label{obser-1}
		\|\sigma(0,\cdot)\|^2_{L^2(0,1)} + 	\|v(0,\cdot)\|^2_{L^2(0,1)}\leq
		M {e^{M(T+\frac{1}{T}+\frac{e^{mMT}}{T^m})}}  \int_0^T\int_{\mathcal{O}} |v(t, x)|^2\, dx\, dt.
	\end{align}
\end{proposition}
\begin{proof}
Let $(\sigma_T, v_T)\in L^2(0,1)\times L^2(0,1)$ and $(\sigma, v)$ be the solution of the adjoint system \eqref{CH adj}. Let us fix $s=s_0, \lambda=\lambda_0$, where $s_0$ and $\lambda_0$ are given in Theorem \ref{Thm_Carleman_main} and the weight functions $\vphi_m, \xi_m$ are as given in \eqref{weight_function} corresponding to $\lambda=\lambda_0$. 

First, note that for some positive constant $C_1$ independent of $T$, we have 
\begin{align}\label{eqobsest1}
0<\vphi_m(t, x)<  \frac{C_1}{2T^{2m}}, \quad \text { in } \left(\frac{T}{4}, \frac{3T}{4}\right)\times (0, 1), \nonumber \\  
e^{-2s_0\vphi_m(t,x)} \geq e^{\frac{-C_1 s_0}{T^{2m}}}\geq e^{-C_1\mu_0{(1+\frac{1}{T}+\frac{e^{mCT}}{T^m})}}, \text { in } \left(\frac{T}{4}, \frac{3T}{4}\right)\times (0, 1), 
\end{align}
since $s_0=\mu_0{(e^{mCT}T^{m}+T^{2m-1}+T^{2m})}$ for some $\mu_0>0$. 
Using the fact that $\xi_m(t, x)\geq \frac{1}{T^{2m}}$ in $Q_T$ along with \eqref{eqobsest1}, for any $s\ge s_0, \lambda\geq \lambda_0$,  we get 
\begin{align}\label{carle-joint1}
	\nonumber&	s_0^3\lambda_0^4 \iintQ e^{-2s_0\vphi_m} \xi^3_m |\sigma_x|^2+ 
	s_0^7\lambda_0^8 \iintQ e^{-2s_0\vphi_m} \xi_m^7 |v|^2 \\ \nonumber & \quad\geq   s_0^3\lambda_0^4 \int_{\frac{T}{4}}^{\frac{3T}{4}}\int_{0}^{1} e^{-2s_0\vphi_m} \xi^3_m |\sigma_x|^2+ 
	s_0^7\lambda_0^8 \int_{\frac{T}{4}}^{\frac{3T}{4}}\int_{0}^{1} e^{-2s_0\vphi_m} \xi_m^7 |v|^2\\
\nonumber	&	\quad \geq e^{-C_1\mu_0{(1+\frac{1}{T}+\frac{e^{mCT}}{T^m})}}\left(\frac{s_0^3\lambda_0^4}{T^{6m}} \int_{\frac{T}{4}}^{\frac{3T}{4}} \int_{0}^{1} |\sigma_x|^2+ 
	\frac{s_0^7\lambda_0^8}{T^{14m}} \int_{\frac{T}{4}}^{\frac{3T}{4}}\int_{0}^{1}   |v|^2\right)\\
 &\quad \geq Ce^{-{C\left(\frac{1}{T}+\frac{e^{mCT}}{T^m}\right)}}\left(\int_{\frac{T}{4}}^{\frac{3T}{4}} \int_{0}^{1} |\sigma_x|^2+\int_{\frac{T}{4}}^{\frac{3T}{4}}\int_{0}^{1}   |v|^2\right),
\end{align}
for some positive constant $C$ independent of $T$, since $s_0\ge \mu_0 T^{2m}$. 
Now using Poincar\'e inequality for the first term of the right hand side of \eqref{carle-joint1} since $\sigma\in L^2(0, T; H^1_0(0, 1))$, from above we obtain 
\begin{align}\label{carle-joint2}
	s_0^3\lambda_0^4 \iintQ e^{-2s_0\vphi_m} \xi^3_m |\sigma_x|^2+ &
	s_0^7\lambda_0^8 \iintQ e^{-2s_0\vphi_m} \xi_m^7 |v|^2 &\geq Ce^{-{C\left(\frac{1}{T}+\frac{e^{mCT}}{T^m}\right)}}\int_{\frac{T}{4}}^{\frac{3T}{4}}\int_{0}^{1} (|\sigma(t,x)|^2+|v(t,x)|^2) dx dt,
\end{align}
for some positive constant $C$ independent of $T$. Next, combining \eqref{aux} and \eqref{carle-joint2}, for a a generic constant $C>0$ independent of $T$, we get 
\begin{equation}\label{aux1}
	\|\sigma(0,\cdot)\|^2_{L^2(0,1)} + 	\|v(0,\cdot)\|^2_{L^2(0,1)}\leq
	\frac{C}{T} {e^{C(T+\frac{1}{T}+\frac{e^{mCT}}{T^m})}}\left(s_0^3\lambda_0^4 \iintQ e^{-2s_0\vphi_m} \xi^3_m |\sigma_x|^2+ 
	s_0^7\lambda_0^8 \iintQ e^{-2s_0\vphi_m} \xi_m^7 |v|^2\right), 
\end{equation}
and finally Carleman estimate \eqref{carleman-joint} along with \eqref{aux1} gives 
\begin{equation}\label{aux2}
	\|\sigma(0,\cdot)\|^2_{L^2(0,1)} + 	\|v(0,\cdot)\|^2_{L^2(0,1)}\leq
	\frac{C}{T} {e^{C(T+\frac{1}{T}+\frac{e^{mCT}}{T^m})}} s_0^{39}\lambda_0^{40} \iint_{(0,T)\times\mathcal{O}} e^{-2s_0\vphi_m} \xi_m^{39} |v|^2.
\end{equation}
Now, noting $C_2:=\min_{x\in[0,1]} \big[e^{\frac{m+1}{m}\lambda_0 k \|\nu\|_{\infty}   }-  e^{\lambda_0\big( k\|\nu\|_{\infty} + \nu(x) \big)}\big]>0$, we get 
$$
s^{39}_0e^{-2s_0\vphi_m(t, x)} \xi_m^{39}(t, x)\le C \frac{s_0^{39}}{t^{39m}(T-t)^{39m}}e^{-\frac{2s_0C_2}{t^m(T-t)^m}}\le C \text { in } Q_T, (\text{ using } x^{39} e^{-x}<{39!} ),
$$
for some generic positive constant $C$, independent of $T$. Using this and again noting $\frac{1}{T}\le e^{\frac{1}{T}}$, from \eqref{aux2}, it follows
	 \begin{equation*}
	 	\|\sigma(0,\cdot)\|^2_{L^2(0,1)} + 	\|v(0,\cdot)\|^2_{L^2(0,1)}\leq M e^{M(T +\frac{1}{T}+\frac{e^{mMT}}{T^m})}\int_0^T\int_{\mathcal{O}} |v|^2\, dx\, dt,
	 	\end{equation*}
for some positive constant $M$ independent of $T$ and hence \Cref{pr ob} is proved. 
\end{proof}

{
\begin{remark}
	Note that, the observability inequality \eqref{obser-1} gives the null controllability of \eqref{CH}--\eqref{in} in $(L^2(0, 1))^2$ at any time $T>0$ using a control $h$ satisfying 
	\begin{equation*}
	\|h\|_{L^2(0, T; L^2(\mathcal{O}))} \leq M e^{M(T +\frac{1}{T}+\frac{e^{mMT}}{T^m})} \|(w_0,\psi_0)\|_{(L^2(0,1))^2}.
	\end{equation*}
However, from above, the null controllability of \eqref{CH}--\eqref{in} can be shown with a control $h$ satisfying \eqref{control_estimate} as in \Cref{Thm-linear}
and that is explained in the next sub-section. The estimate of the control given in \eqref{control_estimate} is more tractable to do the analysis used in the later part of this article. 
\end{remark}
}

\subsection{Proof of \Cref{Thm-linear}}\label{sec-proof}
{
First note that from \Cref{pr ob}, it follows that for any $(w_0, \psi_0)\in (L^2(0, 1))^2$ and any $T>0$, the linear system \eqref{CH}--\eqref{in} is null controllable in $(L^2(0, 1))^2$ at time $T$ with a control $h\in L^2(0, T; L^2(\mathcal{O}))$ satisfying 
\begin{equation}\label{eqcontrol-cost1}
\|h\|_{L^2(0, T; L^2(\mathcal{O}))} \leq  M e^{M(T +\frac{1}{T}+\frac{e^{mMT}}{T^m})} \|(w_0,\psi_0)\|_{(L^2(0,1))^2},
\end{equation}
for some  positive constants $M$ and $m>3$, independent of $T$.

To obtain the cost of control as mentioned in \Cref{Thm-linear}, we split the proof of \Cref{Thm-linear} in to two cases below. 

\noindent 
{\bf{Case 1.}} Let $0<T<1$. Then for any given $m>3$, note that $\frac{1}{T}<\frac{1}{T^m}$ and $e^{mMT}\le e^{mM}$. Using this, from \eqref{eqcontrol-cost1} it easily follows that 
\begin{equation}\label{eqcontrol-cost2}
\|h\|_{L^2(0, T; L^2(\mathcal{O}))} \leq \tilde M e^{\tilde M(T +\frac{1}{T^m})} \|(w_0,\psi_0)\|_{(L^2(0,1))^2},
\end{equation} for some $\tilde M>0$ independent of $T$. Hence \Cref{Thm-linear} is proved for $0<T<1$. 

\noindent
{\bf{Case 2.}} Let $T\ge 1$. Let us consider a fix $T_0\in (0, 1)$. From Case 1, it follows that for any $(w_0, \psi_0)\in (L^2(0, 1))^2$, there exists a control $h\in L^2(0, T_0; L^2(\mathcal{O}))$ satisfying \eqref{eqcontrol-cost2} such that the linear system \eqref{CH}--\eqref{in} admits a unique solution $(w, \psi)\in C([0, T_0]; (L^2(0, 1))^2)$ obeying 
$ (w, \psi)(T_0, \cdot)= (0, 0)$ in $(L^2(0, 1))^2$. Setting the control 
\begin{equation}\label{rev_con_new}
	\widetilde h=\begin{cases} h \text{ in } (0,T_0)\times \mathcal{O},\\ 0 \text{ in } (T_0,T)\times \mathcal{O},
	\end{cases}
\end{equation} 
where $h$ is the control  $L^2(0, T_0; L^2(\mathcal{O}))$ mentioned above, we consider \eqref{CH}--\eqref{in} with the control $\widetilde{h}\in L^2(0, T; L^2(\mathcal{O}))$. Note that the well-posedness of the system gives the existence of the unique solution of \eqref{CH}--\eqref{in} in $C([0, T]; (L^2(0,1))^2)$ and the solution is denoted by $(w, \psi)$. Due to the first part, we already have $(w, \psi)(T_0, \cdot)=(0, 0)$ in $(L^2(0, 1))^2$. Since on $[T_0, T]$, the linear system \eqref{CH}--\eqref{bd} is with control $\widetilde{h}=0$ on $(T_0, T)\times \mathcal{O}$ and with initial condition $(w, \psi)(\cdot, T_0)=(0, 0)$,  the solution on $[T_0, T]$ is $(w, \psi)(\cdot, t)=(0, 0)$ for all $t\in [T_0, T]$. Hence the control $\widetilde{h}\in L^2(0, T; L^2(\mathcal{O}))$ gives the null controllability of 
\eqref{CH}--\eqref{in} in $(L^2(0, 1))^2$ at $T$. Moreover, due to $e^{\tilde M(T +\frac{1}{T^m})}\ge 1$ along with \eqref{eqcontrol-cost2}, the control $\widetilde{h}$ satisfies 
\begin{align*}\|\tilde h\|_{L^2(0, T; L^2(\mathcal{O}))}= \| h\|_{L^2(0, T_0; L^2(\mathcal{O}))}&\leq \tilde M e^{\tilde M(T_0 +\frac{1}{T_0^m})} \|(w_0,\psi_0)\|_{(L^2(0,1))^2}\\
	&\leq \tilde M e^{\tilde M(T_0 +\frac{1}{T_0^m})}e^{\tilde M(T +\frac{1}{T^m})}\|(w_0,\psi_0)\|_{(L^2(0,1))^2}\\
	&\leq  M_1 e^{\tilde M(T +\frac{1}{T^m})}\|(w_0,\psi_0)\|_{(L^2(0,1))^2},\end{align*} 
where $ M_1= \tilde Me^{ \tilde M(T_0 +\frac{1}{T_0^m})}$ is a fixed positive constant independent of $T$. 
Hence \Cref{Thm-linear} is proved. 
}

\section{Null controllability of the linearized model with nonhomogeneous source term}\label{Sec:control with source}
In this section, using the controllability and the control cost in \eqref{control_estimate} for the linear system without forcing term, we shall prove the null controllability of the linearized model with nonhomogeneous source provided the source lies in some suitable weighted spaces. 
The proof will be based on  the  so-called source term method developed in \cite{Tucsnak-nonlinear}.
{It is worth mentioning that in \cite{Tucsnak-nonlinear}, the authors demonstrate the source term method with an associated control cost of the form $ Ce^{C/T} $. In contrast, the work \cite{TTW23} addresses a different type of control cost, more precisely $ Ce^{C/T^l}$ with $l\geq 11$ (see Lemma 4.1 and Section 4.3 in \cite{TTW23} for details), which we adopt in this work to obtain our result.
}

Recall that in \Cref{Thm-linear}, we have the null controllability of \eqref{CH} in $(L^2(0, 1))^2$ at any time $T>0$ with the control cost $ {Me^{M T}e^{\frac{M}{T^m}}}$ given in \eqref{control_estimate} for some $M>0$ independent of $T$. In view of this, we introduce the following weighted spaces to handle the system with forcing terms.

Let us first assume the constants  $p>0$, $q>1$  in such a way that 
\begin{align}\label{choice-p_q}
	1<q<{{2}^{\frac{1}{2m}}}, \ \ \text{and} \ \ p> \frac{{q^{2m}}}{2-{q^{2m}}}.
\end{align}
Let $M$ be the positive constant as obtained in the control cost \eqref{control_estimate}. Let us define the functions 

\begin{equation}\label{def_weight_func}
	\rho_0(t)=\left\{\begin{array}{ll} e^{-\frac{pM}{{(q-1)^m(T-t)^m}}} & t \in [0, T), \\ 0 & t=T, \end{array}\right.   \quad 
\rho_{\F}(t)= \left\{\begin{array}{ll} e^{-\frac{(1+p){{q^{2m}}} M}{{(q-1)^m(T-t)^m}}} & t \in [0, T), \\ 0 & t=T. \end{array}\right.
\end{equation}

Note that the functions $\rho_0$ and $\rho_{\F}$ are continuous and non-increasing in $[0,T]$. 
\begin{remark}\label{Remark-weights}
	We compute that
	\begin{align*}
		\frac{\rho_0^2(t)}{\rho_\F(t)}= {e^{\frac{q^{2m} M + pM(q^{2m}-2)}{(q-1)^m(T-t)^m}}}, \quad \forall t \in \left[0, T\right), \quad \mathrm{with}\quad 
		\frac{\rho_0^2(T)}{\rho_\F(T)}=0.
	\end{align*} 
	Due to the choices of $p,q$ in \eqref{choice-p_q}, we have ${M\big(q^{2m}+ p(q^{2m}-2)\big)}<0$, $(q-1)>0$ and therefore we can conclude that
	\begin{align*}
		\frac{\rho^2_0(t)}{\rho_{\F}(t)} \leq 1, \quad \forall t \in [0,T]. 
	\end{align*}
\end{remark}

We define the following weighted spaces: 
\begin{subequations}
	\begin{align}
		\label{space_F}
		&\F:= \left\{f\in L^1(0,T; L^2(0,1)) \ \Big| \ \frac{f}{\rho_\F} \in L^1(0,T; L^2(0,1))  \right\}, \\
		\label{space_Y} 
	&	\Y := \left\{\begin{array}{l} (u,\phi)\in C([0, T]; (L^2(0,1))^2) \Big| \frac{u}{\rho_0} \in C([0,T]; L^2(0,1)) \cap L^2(0,T; H^1_0(0,1)), \\
	    \hspace{2cm} \frac{\phi}{\rho_0}\in C([0,T]; L^2(0,1)) \cap L^2(0,T; H^2_1) \end{array}\right\}, \\
		\label{space_V}
	&	\V:= \left\{ h\in L^2((0,T) \times \mathcal{O})  \ \Big| \ \frac{h}{\rho_0}\in L^2((0,T) \times \mathcal{O})  \right\}.
	\end{align}
\end{subequations}
Let us define the following
 norms for the above weighted spaces by
 \begin{align*}
 	\|f\|_{\F} := \| \rho^{-1}_\F f\|_{ L^2(0, T; L^2(0,1))}, \quad 
 	\|h\|_{\V} := \|\rho^{-1}_0 h\|_{L^2(0, T; L^2(\mathcal{O})}, \\
 	\|(u, \phi)\|_{\Y}:=\|\rho^{-1}_0 u\|_{C([0,T]; L^2(0,1)) \cap L^2(0,T; H^1_0(0,1))}+ \|\rho^{-1}_0 \phi\|_{C([0,T]; L^2(0,1)) \cap L^2(0,T; H^2_1)}.
 \end{align*}

Consider the following system with source terms $(f_1,f_2)$ in the right hand side of the systems \eqref{CH}:
 \begin{equation}\label{System-source}
 	\begin{cases}
 		w_t-\gamma w_{xx}+\bar u(x) w_x+\bar u'(x) w=\gamma_1 \psi_x+f_1, & \text{in } Q_T, \\
 		\psi_t+ \psi_{xxxx}+\gamma_2 \psi_{xx} + \bar u(x) \psi_x = f_2+\chi_{\mathcal{O}} h, &\text{in } Q_T\\
 		w(t, 0)=0,  \  w(t, 1)=0,  & t \in (0, T), \\
 		\psi_x(t,0)=0, \  \psi_x(t,1)=0, & t \in (0, T),    \\
 		\psi_{xxx}(t, 0)=0,  \  \psi_{xxx}(t, 1)=0, & t \in (0, T),\\
 		w(0, x)=w_0(x),\, \psi(0,x)=\psi_0(x), & x \in  (0, 1).
 	\end{cases}
 \end{equation}
Our next theorem is for the null controllability result of the inhomogeneous linear
system provided that the forcing terms $f_1$ and $f_2$ belong to the space $\mc{S}$ with the weight $\rho_{\mc S}^{-1}$. Thanks to the weight functions, the forcing terms vanish at $t=T$ with the decay $\rho_{\mc S}$. Our weight functions are associated with this control cost.  Utilizing all these, we show that the solution of the inhomogeneous system \eqref{System-source} belongs to the weighted space defined above with the weight $\rho_0$ having a singularity at $t = T$, and hence the solution has to be zero at
$t = T$. To do it, we write the time interval $(0, T )$ as $\cup_{k\geq0}(T_{k} , T_{k+1})$, for $T_k$’s defined
suitably depending on $q$ which is introduced in the weight functions. Then, on each $(T_k , T_{k+1})$, we split the inhomogeneous system \eqref{System-source} into a system with the forcing term and zero initial condition and a control system. Finally, using some properties of the weight functions and their values at $T_k$ for all $k \geq0$, we obtain our required
estimates.

Next, by the following proposition we shall obtain the existence of a control $h\in \V$  for the system \eqref{System-source} with  given source term $f_1, f_2\in \F$ and initial data $(w_0,\psi_0) \in L^2(0,1)\times L^2(0,1)$.

Let us  prove the following result. 
\begin{proposition}\label{Proposition-weighted}
	Let $T>0$.  For any given $f_1,f_2\in \F$ and for any $(w_0,\psi_0) \in (L^2(0,1))^2$, there exists a control $h\in \V$ such that \eqref{System-source} admits a unique solution 
	$(w,\psi)\in \Y$ satisfying 
\begin{equation}\label{eqnullcond} 
	w(T)=0=\psi(T) \quad \mathrm{in}\quad L^2(0, 1).
	\end{equation}
Further, the solution and the control satisfy
	\begin{multline}\label{estimate-weighted}
		\left\|\frac{w}{\rho_0} \right\|_{C([0,T]; L^2(0,1)) \cap L^2(0,T; H^1_0(0,1))} + 	\left\|\frac{\psi}{\rho_0} \right\|_{C([0,T]; L^2(0,1)) \cap L^2(0,T; H^2_1)} + \left\|\frac{h}{\rho_0} \right\|_{L^2((0,T)\times \mathcal{O})}\\
		\leq  C{e^{C(T+\frac{1}{T^m})} } \left(\|(w_0,\psi_0)\|_{L^2(0,1)\times L^2(0,1)} + \left\|{\left(\frac{f_1}{\rho_\F},\frac{f_2}{\rho_\F}\right)}\right\|_{L^1((0,T;L^2(0,1)\times L^2(0,1))} \right),
	\end{multline}
 where the constant $C>0$ does not depend on $w_0,\psi_0$, $f_1,f_2$, $h$, $T$. 
 
\noindent

\end{proposition}

\begin{proof}
Let any $T>0$ be given. Then, we define the sequence $(T_k)_{k\geq 0}$ as 
	\begin{align}\label{def-T-k}
		T_k = T - \frac{T}{q^k}, \qquad \forall k\geq 0,
		\end{align}
	where $q$ is given by \eqref{choice-p_q} and it yields 
	\begin{equation}\label{timedif}
	 0< T_{k+1}-T_k\le T({{2}^{\frac{1}{2m}}}-1), \quad \forall\, k\ge 0.
	\end{equation}
With this $T_k$, it can be shown that the weight functions  $\rho_0$ and $\rho_\F$ enjoy the following relation:
\begin{align}\label{relation_rho_0_rho_F}
	\rho_0(T_{k+1}) =  \rho_{\F}(T_{k-1}) {e^{\frac{M}{(T_{k+1}-T_{k})^m}}} , \qquad \forall k\geq 1.
\end{align}
Here, we mention that throughout the proof of the proposition, $C$ denotes a generic positive constant independent of $T$ and $k$.

\noindent 
Our aim is to decompose \eqref{System-source} on $(T_k, T_{k+1})$ for each $k\ge 0$, into two systems- one is only with the forcing term and zero initial data, and the other system is with control and non-zero initial data as follows. 

\noindent 
For all $k\ge 0$, we consider $\left ( \widetilde w, \widetilde\psi\right )$ is such that
\begin{align*}
	& \widetilde w \in C([T_k,T_{k+1}]; 
	(L^2(0,1)) \cap L^2(T_k,T_{k+1}; H^1_{0}(0,1)), 
	\\
	& \widetilde \psi \in C([T_k,T_{k+1}]; (L^2(0,1)) \cap L^2(T_k,T_{k+1}; H^2_1),
\end{align*}
and that
$(\widetilde w, \widetilde\psi)$ is the unique solution of 
\begin{equation}\label{System-source-no-control}
	\begin{cases}
		\widetilde w_t-\gamma \widetilde w_{xx}+\bar u(x) \widetilde w_x+\bar u'(x) \widetilde w=\gamma_1 \widetilde \psi_x+f_1, & \text{in } Q_k, \\
		\widetilde \psi_t+ \widetilde  \psi_{xxxx}+\gamma_2\widetilde \psi_{xx} + \bar u(x)\widetilde \psi_x = f_2 &\text{in } \in  Q_k,\\
		 	\widetilde w(t, 0)=0,  \ 	\widetilde   w(t, 1)=0,  & t \in (T_k, T_{k+1}), \\
			\widetilde\psi_x(t,0)=0, \  	\widetilde\psi_x(t,1)=0, & t \in (T_k, T_{k+1}),    \\
			\widetilde\psi_{xxx}(t, 0)=0,  \  	\widetilde\psi_{xxx}(t, 1)=0, & t \in (T_k, T_{k+1}),\\
			\widetilde w(T^+_{k},x)=0, \widetilde \psi(T^+_{k},x)=0, & x \in  (0, 1).
	\end{cases}
\end{equation}
where $Q_k=(T_k, T_{k+1})\times (0,1).$ Moreover, for all $k\ge 0$, it satisfies the estimate (using \Cref{prop-esti-regular}) 
\begin{align*}
	\left\|{\widetilde w} \right\|_{C([T_k,T_{k+1}]; 
		(L^2(0,1)) \cap L^2(T_k,T_{k+1}; H^1_{0}(0,1))} + 
	\left\|\widetilde \psi \right\|_{C([T_k,T_{k+1}]; (L^2(0,1)) \cap L^2(T_k,T_{k+1}; H^2_1)}\\
	\leq  Ce^{CT}  
	\left\|(f_1, f_2) \right\|_{L^1(T_k, T_{k+1}; L^2\times L^2)}.
\end{align*}
Denoting the sequence $\{a_k\}_{k\geq 0}$ with
\begin{align}\label{sequence_a_k}
	a_0 = (w_0,\psi_0) \in (L^2(0,1))^2, \quad a_{k+1} =  
	\left ( \widetilde w( T^{-}_{k+1}), \widetilde \psi( T^{-}_{k+1}) 
	\right )\in (L^2(0,1))^2, \,\,\,\, \forall k \geq 0 ,
\end{align}
from the above estimate, we have 
\begin{align}\label{esti-m-k}
	\left\|a_{k+1} \right\|_{L^2\times L^2} \leq  Ce^{CT} \left\|(f_1, f_2) \right\|_{L^1(T_k, T_{k+1}; L^2\times L^2)} , \quad \forall k\geq 0.
\end{align}

\noindent 
Next, for each $k\geq 0$, we consider the system with control $h_k\in L^2(T_k, T_{k+1}; L^2(\mathcal{O}))$ in the time interval $(T_k, T_{k+1})$ 
\begin{equation}\label{eq tk}
	\begin{cases}
	\overline w_t-\gamma \overline w_{xx}+\bar u(x) \overline w_x+\bar u'(x) \overline w=\gamma_1 \overline \psi_x, & \text{in } Q_k, \\
	\overline \psi_t+ \bar  \psi_{xxxx}+\gamma_2\overline \psi_{xx} + \bar u(x)\overline \psi_x = \chi_{\mathcal{O}} h_k &\text{in } \in  Q_k,\\
	\overline w(t, 0)=0,  \ 	\overline   w(t, 1)=0,  & t \in (T_k, T_{k+1}), \\
	\overline\psi_x(t,0)=0, \  	\overline\psi_x(t,1)=0, & t \in (T_k, T_{k+1}),    \\
	\overline\psi_{xxx}(t, 0)=0,  \  	\overline\psi_{xxx}(t, 1)=0, & t \in (T_k, T_{k+1}),\\
	(\overline w(T^+_{k},x), \overline \psi(T^+_{k},x))=a_k, & x \in  (0, 1),
	\end{cases}
\end{equation}
such that $(\overline w(T^-_{k+1},x), \overline \psi(T^-_{k+1},x))=(0, 0)$, for all $x\in (0, 1)$. Due to \Cref{Thm-linear}, such a control $h_k\in L^2(T_k,T_{k+1}; L^2(\mathcal{O}))$
exists and it satisfies
\begin{align}\label{control_estimate tk}
	\|h_k\|_{L^2((T_k, T_{k+1})\times \mathcal{O})} \leq M e^{M(T_{k+1}-T_k)} e^{M/
		{\left (  T_{k+1}-T_k
		\right )^m } }	\left\|a_{k} \right\|_{L^2\times L^2} , \quad \forall k \, \geq 0, 
\end{align}
where the constant
$M>0$ as in \eqref{control_estimate}) neither depends on $k$, $T$ nor on $a_k$.
Combining \eqref{timedif}, \eqref{esti-m-k} and \eqref{control_estimate tk}, for all $k\ge 1$, we have  
\begin{align}\label{hk1}
	\|h_{k}\|_{L^2((T_{k},T_{k+1})\times \mathcal{O})} &\leq Ce^{CT}e^{M/
		{\left (  T_{k+1}-T_{k}
		\right )^m }}\|a_{k}\|_{L^2\times L^2}    \notag \\
	&\leq C e^{CT} e^{M/
		{\left (  T_{k+1}-T_{k}
		\right )^m} }\left\|(f_1,f_2) \right\|_{L^1(T_{k-1}, T_{k}; L^2\times L^2)} \notag \\
	& \leq C e^{CT} {e^{M/(T_{k+1} - T_{k})^m}} \rho_{\F}(T_{k-1}) \left\|{\left(\frac{f_1}{\rho_\F},\frac{f_2}{\rho_\F}\right)}\right\|_{L^1(T_{k-1}, T_{k}; L^2\times L^2)},
\end{align}
since
$\rho_\F$ is a non-increasing function.
Then, using the relation
\eqref{relation_rho_0_rho_F}, one can write 
\begin{equation*}
	\norm{{h}_{k}}_{L^2((T_{k},T_{k+1})\times \mathcal{O})}\leq C e^{CT}\rho_0(T_{k+1})\left\|{\left(\frac{f_1}{\rho_\F},\frac{f_2}{\rho_\F}\right)}\right\|_{L^1(T_{k-1}, T_{k}; L^2\times L^2)}, \quad \forall k \geq 1.  
\end{equation*}
Again, since
$\rho_0$ is also non-increasing, we deduce that 
\begin{align}\label{esti-contr-k}
	\norm{\frac{{h}_{k}}{\rho_0}}_{L^2((T_{k},T_{k+1})\times \mathcal{O})} & \leq\frac{1}{\rho_0(T_{k+1})}\norm{{h}_{k}}_{L^2((T_{k},T_{k+1})\times \mathcal{O})}  \notag \\
	&\leq C e^{CT}\left\|{\left(\frac{f_1}{\rho_\F},\frac{f_2}{\rho_\F}\right)}\right\|_{L^1(T_{k-1}, T_{k}; L^2\times L^2)}, \text{ for all } k \geq 1,
\end{align} 
and for $k=0$, from \eqref{control_estimate tk}, using the definitions of $a_0,$ $\rho_0(T_1)$ and $T_1$, we get 
\begin{align}\label{esti-h_0}
	\norm{\frac{{h}_0}{\rho_0}}_{L^2((0,T_1)\times \mathcal{O})}\le 
	\frac{1}{\rho_0(T_1)}\norm{{h}_0}_{L^2((0,T_1)\times \mathcal{O})} & \leq Me^{M T_1}{e^{M/T_1^m}}\frac{1}{\rho_0(T_1)}\norm{(w_0,\psi_0)}_{L^2\times L^2}  \notag \\
	& \leq C{e^{C(T+\frac{1}{T^m})}} \norm{(w_0,\psi_0)}_{L^2\times L^2}. 
\end{align}
We now define the control function
${h}$ as
\begin{equation*}
	h:=\sum_{k\geq 0} {h}_k \chi_{(T_k,T_{k+1})},   \quad \text{in } \ (0,T)\times \mathcal{O}, 
\end{equation*}
where
$\chi$ denotes the characteristic function. Combining  the estimates \eqref{esti-h_0} and \eqref{esti-contr-k}, we obtain 
\begin{align}\label{wght_ctl_est}
	\norm{\frac{h}{\rho_0}}_{L^2((0,T)\times \mathcal{O})}\leq C{e^{C(T+\frac{1}{T^m})}}
	\left( \norm{(w_0,\psi_0)}_{L^2\times L^2}+\left\|{\left(\frac{f_1}{\rho_\F},\frac{f_2}{\rho_\F}\right)}\right\|_{L^1(0, T; L^2\times L^2)}\right),
\end{align}
for some constant $C>0$.

\noindent 
Let us set $ (w,\psi)=(\widetilde{w},\widetilde{\psi})+(\overline{w},\overline{\psi})$,
where $(\widetilde{w},\widetilde{\psi})$ and $(\overline{w},\overline{\psi})$ are the solutions of \eqref{System-source-no-control} and \eqref{eq tk}, respectively. 
We note that $(w,\psi)$ satisfies \eqref{System-source} on $(T_k, T_{k+1})$ for each $k\ge 0$, with $(w(T_0, \psi(T_0))=(w_0, \psi_0)\in (L^2(0,1))^2$. Since 
for each $k\ge0$, we have $\left(\widetilde{w}(T^+_{k+1}), \widetilde{\psi}(T^+_{k+1})\right)=(0, 0)$ and $\left( \overline{w}(T^-_{k+1}), \overline{\psi}(T^-_{k+1})\right )=(0, 0)$,
we have,  for all
$k\geq0$, that
\begin{align*}
	&	
	\left ( w(T^-_{k+1}), \psi(T^-_{k+1})
	\right ) =
	\left ( \widetilde{w}(T^-_{k+1}), \widetilde{\psi}(T^-_{k+1}) 
	\right ) + 
	\left ( \overline{w}(T^-_{k+1}), \overline{\psi}(T^-_{k+1})
	\right )  = a_{k+1}, \\
	&	
	\left ( w(T^+_{k+1}), \psi(T^+_{k+1})
	\right ) =
	\left ( \widetilde{w}(T^+_{k+1}), \widetilde{\psi}(T^+_{k+1}) 
	\right ) + 
	\left ( \overline{w}(T^+_{k+1}), \overline{\psi}(T^+_{k+1})
	\right )  = a_{k+1}. 
\end{align*}
This gives the continuity of the component $(w, \psi)$ at each $T_{k+1}$ for $k\geq 0$, more precisely $  (w,\psi) \in  C([T_{k},T_{k+1}]; (L^2(0,1)\times L^2(0,1))).$
Consequently, $w\in C([0,T]; L^2(0,1)) \cap L^2(0,T; H^1_0(0,1))$ and $\psi\in C([0,T]; L^2(0,1)) \cap L^2(0,T; H^2_1)$ and $(w, \psi)$ satisfies \eqref{System-source} on $(0, T)$. Moreover, we have the following estimate (thanks to \Cref{thwellposed} and \Cref{prop-esti-regular}) for all
$k\geq 0$: 
\begin{align}\label{esti-sol-t-k}
\left\|{w} \right\|_{C([T_k,T_{k+1}]; 
	(L^2(0,1)) \cap L^2(T_k,T_{k+1}; H^1_{0}(0,1))} + 
\left\| \psi \right\|_{C([T_k,T_{k+1}]; (L^2(0,1)) \cap L^2(T_k,T_{k+1}; H^2_1(0,1))}\\ 
	\nonumber	\leq  Ce^{CT}
	\left(   \norm{a_{k}}_{L^2\times L^2}+\norm{{h}_{k}}_{L^2((T_{k},T_{k+1})\times \mathcal{O})}+ \left\|(f_1, f_2) \right\|_{L^1(T_k, T_{k+1}; L^2\times L^2)}\right)  
\end{align}
We start with
$k\geq 1$; using the estimates of
$a_k$ and
$h_k$ from
\eqref{esti-m-k} and
\eqref{control_estimate tk} (resp.), we get 
\begin{align*} 
	&\left\|{w} \right\|_{C([T_k,T_{k+1}]; 
		(L^2(0,1)) \cap L^2(T_k,T_{k+1}; H^1_{0}(0,1))} + 
	\left\| \psi \right\|_{C([T_k,T_{k+1}]; (L^2(0,1)) \cap L^2(T_k,T_{k+1}; H^2_1(0,1))}\\
	&\leq Ce^{CT} {e^{\frac{M}{(T_{k+1}-T_{k})^m}}}\left\|(f_1, f_2) \right\|_{L^1(T_{k-1}, T_k; L^2\times L^2)}+ C e^{CT} \left\|(f_1, f_2) \right\|_{L^1(T_k, T_{k+1}; L^2\times L^2)}\\
	&\leq Ce^{CT} \rho_\F(T_{k-1}) {e^{\frac{M}{(T_{k+1}-T_{k})^m}}}\left\|{\left(\frac{f_1}{\rho_\F},\frac{f_2}{\rho_\F}\right)}\right\|_{L^1(T_{k-1}, T_{k+1}; L^2\times L^2)},
\end{align*}
since
$\rho_{\F}$ is non-increasing in
$(0,T)$, and then \eqref{relation_rho_0_rho_F} yields 
\begin{align}\label{esti-sol-k-k+1}
	& \norm{\frac{w}{\rho_0}}_{C([T_k,T_{k+1}]; 
		(L^2(0,1)) \cap L^2(T_k,T_{k+1}; H^1_{0}(0,1))  }
	+ \norm{\frac{\psi}{\rho_0}}_{C([T_k,T_{k+1}]; 
		(L^2(0,1)) \cap L^2(T_k,T_{k+1}; H^2_{1}(0,1))}  \notag \\
	&\hspace{5.5cm} \leq   Ce^{CT} \left\|{\left(\frac{f_1}{\rho_\F},\frac{f_2}{\rho_\F}\right)}\right\|_{L^1(T_{k-1}, T_{k+1}; L^2\times L^2)}, \quad \forall k\geq 1.
\end{align}
Now, for $k=0$, i.e., on $(0, T_1)$, using \eqref{esti-h_0} and the fact that $\rho_\F(t)<1$ on $(0, T_1)$, from \eqref{esti-sol-t-k}, we get 
\begin{align}\label{esti-sol-0-T}
\notag	& \norm{\frac{w}{\rho_0}}_{C([0,T_1]; 
		(L^2(0,1)) \cap L^2(0,T_1; H^1_{0}(0,1)) }
	+ \norm{\frac{\psi}{\rho_0}}_{C([0,T_1]; 
		(L^2(0,1)) \cap L^2(0,T_1; H^2_{1}(0,1))} \\
	& \leq    C {e^{CT+\frac{C}{T^m}} }  
	\left ( \|(w_0, \psi_0)\|_{L^2\times L^2} + \left\|{\left(\frac{f_1}{\rho_\F},\frac{f_2}{\rho_\F}\right)}\right\|_{L^1(0, T_1; L^2\times L^2)}
	\right ). 
\end{align}
Combining \eqref{esti-sol-0-T}, \eqref{esti-sol-k-k+1} and \eqref{wght_ctl_est}, we obtain \eqref{estimate-weighted}. Moreover, since $w(T, \cdot)$ and $\psi(T, \cdot)$ are defined and belong to $L^2(0, 1)$, and $\rho_0^{-1}(t)\rightarrow \infty$ as $t\rightarrow T$, it yields \eqref{eqnullcond}.

This completes the proof.

\end{proof}

\section{Local null controllability of the nonlinear system}\label{Sec:Control nonlin}
This section is devoted to the proof of our main results \Cref{Thm-nonlinear}. In the previous section we have proved that linearized model \eqref{System-source} is null controllable whenever the source terms lie in the weighted space $\F$. Using this result along with Banach fixed point theorem, we conclude the main result.

Let us recall the weighted spaces \eqref{space_F}, \eqref{space_Y} and \eqref{space_V}. 

First, we obtain the estimates of $N_1$ and $N_2$ in the weighted space. 
\begin{lemma}\label{non est}
Let $N_1$ and $N_2$ be given by \eqref{nonlinear}. Then for any $(u, \varphi)\in \Y$, $N_i(u, \varphi)\in \F$, for $i=1, 2$. Further, there exists a positive constant $C_1$ 
depending only on the coefficients in $N_1$ and $N_2$ such that 
        \begin{align}\label{non est1}
            \|N_i(u,\varphi)\|_{\F}\leq C_1(\|(u, \varphi)\|^4_{\Y}+\|(u, \varphi)\|^3_{\Y}+\|(u, \varphi)\|^2_{\Y}), \quad  i=1,2,  \quad \forall\, (u, \varphi)\in \Y,
        \end{align}
     and for all $(u_1, \varphi_1), (u_2, \varphi_2)\in \Y$, 
        \begin{align}\label{Lipnl}
          \|N_i(u_1,\varphi_1)-N_i(u_2,\varphi_2)\|_{\F} \leq &C_1\sum_{j=1}^2\Big(\|(u_j, \varphi_j)\|^3_{\Y}+\|(u_j, \varphi_j)\|^2_{\Y}\nonumber\\
          &+\|(u_j, \varphi_j)\|_{\Y}\Big)\|(u_1-u_2, \varphi_1-\varphi_2)\|_{\Y}, \quad i=1,2.     
          \end{align}
\end{lemma}
\begin{proof}
In view of \Cref{Remark-weights} and \eqref{def_weight_func}, we first note that 
\begin{equation}\label{ratioweight}
\left|\frac{\rho^2_0}{\rho_{\F}}\right|<1, \quad \left|\frac{\rho^3_0}{\rho_{\F}}\right|<1, \quad \left|\frac{\rho^4_0}{\rho_{\F}}\right|<1, \quad \mathrm{on}\quad (0, T)\times (0, 1). 
\end{equation}
In the sequel, $C_1$ denotes a generic positive constant which may depend on the coefficients in $N_1$ and $N_2$. 

Recall the definition of $N_1$ and $N_2$ as given in \eqref{nonlinear}. 
For any $(u, \varphi)\in \Y$, we have 
\begin{align}\label{eqN1}
\norm{\frac{N_1(u,\varphi)}{\rho_{\F}}}_{L^1(L^2)}\leq C_1\int_{0}^{T}\frac{1}{\rho_{\F}(t)}\Big(\norm{u(t, \cdot) u_x(t, \cdot)}_{L^2}+\norm{\varphi_x(t, \cdot)\varphi_{xx}(t, \cdot)}_{L^2} \notag \\ 
+\norm{\varphi^3(t, \cdot)\varphi_x(t, \cdot)}_{L^2}+\norm{\varphi^2(t, \cdot)\varphi_x(t, \cdot)}_{L^2}+\norm{\varphi(t, \cdot)\varphi_x(t, \cdot)}_{L^2}\Big)\, dt
\end{align}
Using the continuous embedding of $H^1(0, 1)$ in $C([0, 1])$ in one-dimensional space and Cauchy-Schwarz inequality along with \eqref{ratioweight}, for any $(u, \varphi)\in \Y$, we estimate the first, second and fifth term in the right hand side above as follows: 
$$ 
\int_{0}^{T}\frac{1}{\rho_{\F}(t)}\norm{u(t, \cdot) u_x(t, \cdot)}_{L^2}\, dt \le \int_0^T \frac{\rho^2_0(t)}{\rho_\F(t)} \norm{\frac{u(t, \cdot)}{\rho_0(t)}}_{L^\infty}\norm{\frac{u(t, \cdot)}{\rho_0(t)}}_{L^2}\, dt\le C_1\norm{\frac{u}{\rho_0}}^2_{L^2( H^1)},
$$
$$ 
\int_{0}^{T}\frac{1}{\rho_{\F}(t)}\norm{\varphi_x(t, \cdot)\varphi_{xx}(t, \cdot)}_{L^2}\, dt \le \int_0^T \frac{\rho^2_0(t)}{\rho_\F(t)} \norm{\frac{\varphi_x(t, \cdot)}{\rho_0(t)}}_{L^\infty}\norm{\frac{\varphi_{xx}(t, \cdot)}{\rho_0(t)}}_{L^2}\, dt\le C_1\norm{\frac{\varphi}{\rho_0}}^2_{L^2( H^2)},
$$
$$ 
\int_{0}^{T}\frac{1}{\rho_{\F}(t)}\norm{\varphi(t, \cdot)\varphi_x(t, \cdot)}_{L^2}\, dt \le \int_0^T \frac{\rho^2_0(t)}{\rho_\F(t)} \norm{\frac{\varphi(t, \cdot)}{\rho_0(t)}}_{L^\infty}\norm{\frac{\varphi_x(t, \cdot)}{\rho_0(t)}}_{L^2}\, dt\le C_1\norm{\frac{\varphi}{\rho_0}}^2_{L^2( H^1)}.
$$
To estimate the third and fourth terms in \eqref{eqN1}, note that using interpolation argument, for any $\varphi\in C([0,T]; L^2(0,1)) \cap L^2(0,T; H^2_1(0,1)),$ we have
 $$\norm{\varphi(t,\cdot)}_{H^1(0,1)}\leq C_1 \norm{\varphi(t,\cdot)}^{1/2}_{H^2(0,1)} \norm{\varphi(t,\cdot)}^{1/2}_{L^2(0,1)}, \quad \forall\, t\in (0, T),$$
 for some positive constant $C_1$ independent of $t$ and $\varphi$ and hence  this implies that
\begin{align}\label{interpolineq}
	\norm{\varphi}_{L^4(0,T;H^1(0,1))}\leq C_1 \norm{\varphi}^{1/2}_{C([0,T];L^2(0,1))}\norm{\varphi}^{1/2}_{L^2(0,T;H^2(0,1))}.
\end{align} 
Using the continuous embedding of $H^1(0, 1)$ in $C([0, 1])$ in one-dimensional space and \eqref{interpolineq} for $\frac{\varphi}{\rho_0}\in C([0, T]; L^2(0,1))\cap L^2(0, 1; H^2_1)$, along with \eqref{ratioweight}, we obtain the estimate of the third term in \eqref{eqN1}:
$$
\begin{array}{l}
\displaystyle \int_0^T \frac{1}{\rho_\F(t)}\norm{\varphi_x(t, \cdot) \varphi^3(t, \cdot)}_{L^2}\, dt\leq
 \int_{0}^{T}\frac{\rho^4_0(t)}{\rho_\F(t)}\norm{\frac{\varphi(t, \cdot)}{\rho_0(t)}}^3_{L^{\infty}}\norm{\frac{\varphi_{x}(t, \cdot)}{\rho_0(t)}}_{L^2}\, dt\\[2.mm]
 \displaystyle \leq C_1 \int_{0}^{T}\norm{\frac{\varphi(t, \cdot)}{\rho_0(t)}}^3_{H^1}\norm{\frac{\varphi_{x}(t, \cdot)}{\rho_0(t)}}_{L^2}\, dt
\leq C_1 \norm{\frac{\varphi}{\rho_0}}^4_{L^4(H^1(0,1))}\leq C_1 \norm{\frac{\varphi}{\rho_0}}^2_{C([0, T]; L^2)}\norm{\frac{\varphi}{\rho_0}}^2_{L^2(H^2)}, 
\end{array}
$$
and the fourth term in \eqref{eqN1}
$$
\begin{array}{l}
\displaystyle \int_0^T \frac{1}{\rho_\F(t)}\norm{\varphi^2(t, \cdot)\varphi_x(t, \cdot)}_{L^2}\, dt\leq
 \int_{0}^{T}\frac{\rho^3_0(t)}{\rho_\F(t)}\norm{\frac{\varphi(t, \cdot)}{\rho_0(t)}}^2_{L^{\infty}}\norm{\frac{\varphi_{x}(t, \cdot)}{\rho_0(t)}}_{L^2}\, dt  \\[2.mm]
 \displaystyle 
 \leq C_1 \int_{0}^{T}\norm{\frac{\varphi(t, \cdot)}{\rho_0(t)}}^2_{H^1}\norm{\frac{\varphi(t, \cdot)}{\rho_0(t)}}_{H^2}\, dt
\leq C_1 \norm{\frac{\varphi}{\rho_0}}^2_{L^4(H^1)}\norm{\frac{\varphi}{\rho_0}}_{L^2(H^2)}
\leq C_1 \norm{\frac{\varphi}{\rho_0}}_{C([0, T]; L^2)}\norm{\frac{\varphi}{\rho_0}}^2_{L^2(H^2)}.
\end{array}
$$
Similarly, $\norm{\frac{N_2(u,\varphi)}{\rho_{\F}}}_{L^1(L^2)}$ can be estimated for any $(u, \varphi)\in \Y$. 
Finally, using above estimates and Poincar\'e inequality for $H^1_0(0, 1)$ along with \eqref{ineq1} for $H^2_1$, we obtain \eqref{non est1}.

To obtain the Lipschitz estimate \eqref{Lipnl}, we note that for any $(u_1, \varphi_2), (u_2, \varphi_2)\in \Y$, 
\begin{align*}
	\norm{\frac{N_1(u_1,\varphi_1)-N_1(u_2,\varphi_2)}{\rho_{\F}}}_{L^1(L^2)}&\leq C_1\bigg[\int_{0}^{T}\frac{1}{\rho_{\F}}\big[\norm{u_1u_{1,x}-u_2u_{2,x}}_{L^2}+\norm{\varphi_{1,x}\varphi_{1,xx}-\varphi_{2,x}\varphi_{2,xx}}_{L^2}\\
	&+\norm{\varphi^3_{1}\varphi_{1,x}-\varphi^3_{2}\varphi_{2,x}}_{L^2}+\norm{\varphi^2_{1}\varphi_{1,x}-\varphi^2_{2}\varphi_{2,x}}_{L^2}+\norm{\varphi_{1}\varphi_{1,x}-\varphi_{2}\varphi_{2,x}}_{L^2}\big]\, dt.
	\bigg]
\end{align*}
Using the similar argument as above, we estimate each term as follows:
\begin{align*}\bullet\int_{0}^{T}\frac{1}{\rho_{\F}} \norm{u_1u_{1,x}-u_2u_{2,x}}_{L^2}\, dt &=\int_{0}^{T}\frac{1}{\rho_{\F}} \norm{u_1(u_{1,x}-u_{2,x})+(u_1-u_2)u_{2,x}}_{L^2}\, dt\\
	&
	\leq C_1 \int_{0}^{T}\frac{\rho_0^2}{\rho_{\F}} \norm{\frac{u_1-u_2}{\rho_0}}_{H^1}\left(\norm{\frac{u_1}{\rho_0}}_{H^1}+\norm{\frac{u_2}{\rho_0}}_{H^1}\right) \, dt\\
	&\leq C_1 \norm{\frac{u_1-u_2}{\rho_0}}_{L^2(H^1)}\left(\norm{\frac{u_1}{\rho_0}}_{L^2(H^1)}+\norm{\frac{u_2}{\rho_0}}_{L^2(H^1)}\right),
	\end{align*}

\begin{align*}\bullet\int_{0}^{T}\frac{1}{\rho_{\F}} \norm{\varphi_{1,x}\varphi_{1,xx}-\varphi_{2,x}\varphi_{2,xx}}_{L^2}\, dt
&=\int_{0}^{T}\frac{1}{\rho_{\F}} \norm{\varphi_{1,x}(\varphi_{1,xx}-\varphi_{2,xx})+(\varphi_1-\varphi_2)_x \varphi_{2,xx}}_{L^2}\, dt\\
	&\leq C_1 \int_{0}^{T}\frac{\rho_0^2}{\rho_{\F}} \norm{\frac{\varphi_1-\varphi_2}{\rho_0}}_{H^2}\left(\norm{\frac{\varphi_1}{\rho_0}}_{H^2}+\norm{\frac{\varphi_2}{\rho_0}}_{H^2}\right)\, dt\\
&\leq C_1 \norm{\frac{\varphi_1-\varphi_2}{\rho_0}}_{L^2(H^2)}\left(\norm{\frac{\varphi_1}{\rho_0}}_{L^2(H^2)}+\norm{\frac{\varphi_2}{\rho_0}}_{L^2(H^2)}\right),
\end{align*}

\begin{align*}\bullet\int_{0}^{T}\frac{1}{\rho_{\F}} \norm{\varphi_{1,x}\varphi^3_{1}-\varphi_{2,x}\varphi^3_{2}}_{L^2}\, dt
&=\int_{0}^{T}\frac{1}{\rho_{\F}} \norm{\varphi^3_{1}(\varphi_{1,x}-\varphi_{2,x})+(\varphi^3_1-\varphi^3_2) \varphi_{2,x}}_{L^2}\, dt\\
	&
	\leq C_1 \Bigg(\int_{0}^{T}\frac{\rho_0^4}{\rho_{\F}} \norm{\frac{\varphi_1-\varphi_2}{\rho_0}}_{H^2}\left(\norm{\frac{\varphi_1}{\rho_0}}^2_{H^1}\norm{\frac{\varphi_1}{\rho_0}}_{L^2}\right) \, dt\\
	&+\int_{0}^{T}\frac{\rho_0^4}{\rho_{\F}} \norm{\frac{\varphi_1-\varphi_2}{\rho_0}}_{L^2}\norm{\frac{\varphi_2}{\rho_0}}_{H^2}\left(\norm{\frac{\varphi_1}{\rho_0}}^2_{H^1}+\norm{\frac{\varphi_2}{\rho_0}}^2_{H^1}\right)\, dt\Bigg)\\
	&\leq C_1 \Bigg(\norm{\frac{\varphi_1}{\rho_0}}_{C([0,T]; L^2)}\norm{\frac{\varphi_1-\varphi_2}{\rho_0}}_{L^2(H^2)}\norm{\frac{\varphi_1}{\rho_0}}^2_{L^4(H^1)}\\
	&+\norm{\frac{\varphi_1-\varphi_2}{\rho_0}}_{C([0, T]; L^2)}\norm{\frac{\varphi_2}{\rho_0}}_{L^2(H^2)}\left(\norm{\frac{\varphi_1}{\rho_0}}^2_{L^4(H^1)}+ \norm{\frac{\varphi_2}{\rho_0}}^2_{L^4(H^1)}\right)\Bigg)\\
	&\hspace{-2.5cm}\leq C_1 \Bigg(\norm{\frac{\varphi_1}{\rho_0}}^2_{C([0, T]; L^2)}\norm{\frac{\varphi_1-\varphi_2}{\rho_0}}_{L^2(H^2)}\norm{\frac{\varphi_1}{\rho_0}}_{L^2(H^2)}+\norm{\frac{\varphi_1-\varphi_2}{\rho_0}}_{C([0, T];L^2)}\norm{\frac{\varphi_2}{\rho_0}}_{L^2(H^2)}
	\\
	&\left(\norm{\frac{\varphi_1}{\rho_0}}_{C([0, T]; L^2)}\norm{\frac{\varphi_1}{\rho_0}}_{L^2(H^2)}+\norm{\frac{\varphi_2}{\rho_0}}_{C([0, T]; L^2)}\norm{\frac{\varphi_2}{\rho_0}}_{L^2(H^2)}\right)\Bigg),
	 \end{align*}
	 
\begin{align*}\bullet\int_{0}^{T}\frac{1}{\rho_{\F}} \norm{\varphi_{1,x}\varphi^2_{1}-\varphi_{2,x}\varphi^2_{2}}_{L^2}\, dt
&\leq C_1 \Bigg(\norm{\frac{\varphi_1}{\rho_0}}_{C([0, T]; L^2)}\norm{\frac{\varphi_1-\varphi_2}{\rho_0}}_{L^2(H^2)}\norm{\frac{\varphi_1}{\rho_0}}_{L^2(H^2)}\\
&+\norm{\frac{\varphi_1-\varphi_2}{\rho_0}}_{C([0, T]; L^2)}\norm{\frac{\varphi_2}{\rho_0}}_{L^2(H^2)} \left(\norm{\varphi_1}_{L^2(H^1)}+\norm{\varphi_2}_{L^2(H^1)}\right)\Bigg).
\end{align*}

\begin{align*}\bullet\int_{0}^{T}\frac{1}{\rho_{\F}} \norm{\varphi_1\varphi_{1,x}-\varphi_2\varphi_{2,x}}_{L^2}\, dt
\leq C_1 \norm{\frac{\varphi_1-\varphi_2}{\rho_0}}_{L^2(H^1)}\left(\norm{\frac{\varphi_1}{\rho_0}}_{L^2(H^1)}+\norm{\frac{\varphi_2}{\rho_0}}_{L^2(H^1)}\right). 
\end{align*}

Similarly, $\norm{\frac{N_2(u_1,\varphi_1)-N_2(u_2,\varphi_2)}{\rho_{\F}}}_{L^1(L^2)}$ can be estimated for any $(u_1, \varphi_1), (u_2, \varphi_2)\in \Y$. 
Again, using Poincar\'e inequality for $H^1_0(0, 1)$ along with \eqref{ineq1} for $H^2_1$, from above estimates we obtain \eqref{Lipnl}.
\end{proof}

Now we are ready to use a fixed point argument to prove \Cref{Thm-nonlinear}. To do it, for any given $(u,\varphi)\in \Y$, we consider a family of the linearized systems with a control $h$ and with forcing terms $N_1(u, \varphi)$ and $N_2(u, \varphi)$, defined in \eqref{nonlinear}
    \begin{equation}\label{lin-sourse-vary}
		\begin{cases}
			w_t-\gamma w_{xx}+\bar u(x) w_x+\bar u'(x) w=\gamma_1 \psi_x+N_1(u,\varphi), & \text{in } Q_T, \\
			\psi_t+ \psi_{xxxx}+\gamma_2 \psi_{xx} + \bar u(x) \psi_x = N_2(u,\varphi)+\chi_{\mathcal{O}} h,  & \text{in } Q_T,\\
			w(t, 0)=0,  \  w(t, 1)=0,  & t \in (0, T), \\
			\psi_x(t,0)=0, \  \psi_x(t,1)=0, & t \in (0, T),    \\
			\psi_{xxx}(t, 0)=0,  \  \psi_{xxx}(t, 1)=0, & t \in (0, T),\\
			w(0, x)=w_0(x),\, \psi(0,x)=\psi_0(x), & x \in  (0, 1).
		\end{cases}
	\end{equation} 
Recall that due to \Cref{Proposition-weighted} and \Cref{non est}, for any given $(u,\varphi)\in \Y$, there exists a control $h^{(u, \varphi)}\in \V$ such that \eqref{lin-sourse-vary} admits a unique solution $(w^{(u, \varphi)}, \psi^{(u, \varphi)})\in \Y$ satisfying the null controllability condition \eqref{eqnullcond} and the estimate \eqref{estimate-weighted} where $f_1=N_1(u, \varphi)$ and $f_2=N_2(u, \varphi)$. 

In view of that, for $R>0$, we set 
$$ \Y_R=\{(u,\varphi)\in \Y : \|(u,\varphi)\|_{\Y}\leq R\},$$
and for any given $(u,\varphi)\in \Y_R$ we define a map $\Lambda$ by 
\begin{equation}\label{eqfixed}
\Lambda(u,\varphi) = (w^{(u, \varphi)}, \psi^{(u, \varphi)}),
\end{equation}
where $(w^{(u, \varphi)}, \psi^{(u, \varphi)})$ is the solution of \eqref{lin-sourse-vary} with the control $h^{(u, \varphi)}$.

Our aim is to show that for some $R>0$, for any $(u,\varphi)\in \Y_R$, $\Lambda(u,\varphi)\in Y_R$ and the map $\Lambda: \Y_R\rightarrow \Y_R$ is a contraction. Then the Banach fixed point theorem gives \Cref{Thm-nonlinear}. 

\bigskip 
\noindent 
\textbf{Proof of \Cref{Thm-nonlinear}.} 
Let us consider $(w_0, \psi_0)\in (L^2(0, 1))^2$ satisfying 
\begin{equation}\label{eqiniest}
 \|(w_0, \psi_0)\|_{L^2\times L^2}\le \mu,
 \end{equation}
where $\mu=\iota R$, for some positive constants $\iota$ and $R$ which will be determined later such that  the map $\Lambda$, defined in \eqref{eqfixed}, becomes a contraction map from 
$\Y_R$ to $\Y_R$ and the fixed point argument goes through. 

Recall that due to \Cref{non est}, for any $(u,\varphi)\in \Y$, $N_i(u, \varphi)\in \F$ for $i=1, 2$ satisfying \eqref{non est1}. 
From \Cref{Proposition-weighted}, it follows that for any $(u,\varphi)\in \Y$, there exists a control $h^{(u, \varphi)}\in \V$ such that \eqref{lin-sourse-vary} admits a unique solution 
$(w^{(u, \varphi)}, \psi^{(u, \varphi)})\in \Y$ satisfying 
$$ \|(w^{(u, \varphi)}, \psi^{(u, \varphi))}\|_{\Y}+\|h^{(u, \varphi)}\|_{\V}\le C{e^{C(T+\frac{1}{T^m})}}  
\left(\|(w_0,\psi_0)\|_{L^2(0,1)\times L^2(0,1)} + \|N_1(u, \varphi)\|_{\F}+\|N_2(u, \varphi)\|_{\F}\right).$$
For $(u,\varphi)\in \Y_R$, from the above inequality along with \eqref{non est1}, we get 
$$
\begin{array}{l}
\|(w^{(u, \varphi)}, \psi^{(u, \varphi)})\|_{\Y}+\|h^{(u, \varphi)}\|_{\V}\le {Ce^{C(T+\frac{1}{T^m})}}\left(\mu + 2C_1(R^4+ R^3+R^2)\right)\\
\hspace{3cm} \le C{e^{C(T+\frac{1}{T^m})}}\left(\iota+ 2C_1(R^3+R^2+R)\right)R\le R,
\end{array}
$$
by choosing 
\begin{equation}\label{eqfixedconstant}
\iota=\frac{1}{4C{e^{C(T+\frac{1}{T^m})}}}, \quad R^3+R^2+R= \frac{1}{8 C_1C{e^{C(T+\frac{1}{T^m})}}}, \quad \mu=\iota R.
\end{equation}
Thus, we have shown that for $\mu, R$ introduced in \eqref{eqfixedconstant}, $\Lambda: \Y_R\rightarrow \Y_R$ is defined. 

Now to show that the map is a contraction, we consider $(u_i, \varphi_i)\in \Y_R$, and $(w^{(u_i, \varphi_i)}, \psi^{(u_i, \varphi_i)})\in \Y_R$ is the solution of \eqref{lin-sourse-vary} with control $h^{(u_i, \varphi_i)}\in \V$ as obtained in \Cref{Proposition-weighted}, for each $i=1, 2$. Denoting $\widehat{w}=w^{(u_1, \varphi_1)}-w^{(u_2, \varphi_2)}$, 
$\widehat{\psi}=\psi^{(u_1, \varphi_1)}-\psi^{(u_2, \varphi_2)}$ and $\widehat{h}=h^{(u_1, \varphi_1)}-h^{(u_2, \varphi_2)}$, note that , 
$\widehat{h}\in \V$ and $(\widehat{w}, \widehat{\psi})$ satisfies 
\begin{equation}\label{lin-sourse-Lip}
		\begin{cases}
			\widehat{w}_t-\gamma \widehat{w}_{xx}+\bar u(x) \widehat{w}_x+\bar u'(x) \widehat{w}=\gamma_1 \widehat{\psi}_x+N_1(u_1,\varphi_1)-N_1(u_2, \varphi_2), & \text{in } Q_T, \\
			\widehat{\psi}_t+ \widehat{\psi}_{xxxx}+\gamma_2 \widehat{\psi}_{xx} + \bar u(x) \widehat{\psi}_x = N_2(u_1,\varphi_1)-N_2(u_2,\varphi_2)+\chi_{\mathcal{O}} \widehat{h},  & \text{in } Q_T,\\
			\widehat{w}(t, 0)=0,  \  \widehat{w}(t, 1)=0,  & t \in (0, T), \\
			\widehat{\psi}_x(t,0)=0, \  \widehat{\psi}_x(t,1)=0, & t \in (0, T),    \\
			\widehat{\psi}_{xxx}(t, 0)=0,  \  \widehat{\psi}_{xxx}(t, 1)=0, & t \in (0, T),\\
			\widehat{w}(0, x)=0,\, \widehat{\psi}(0,x)=0, & x \in  (0, 1).
		\end{cases}
	\end{equation} 
Thus, using \eqref{estimate-weighted} and \eqref{Lipnl}, we get 
$$ 
\begin{array}{l}
\|\Lambda(u_1, \varphi_1)-\Lambda(u_2, \varphi_2)\|_\Y +\|h^{(u_1, \varphi_1)}-h^{(u_2, \varphi_2)}\|_{\V}\\
\le {Ce^{C(T+\frac{1}{T^m})}} \Big(\|N_1(u_1, \varphi_1)- N_1(u_2, \varphi_2)\|_{\F}+ \|N_2(u_1, \varphi_1)- N_2(u_2, \varphi_2)\|_{\F}\Big)\\
\le 4C_1C{e^{C(T+\frac{1}{T^m})}}\left(R^3+R^2+R\right)\|(u_1-\varphi_1, u_2-\varphi_2)\|_{\Y} \le \frac{1}{2}\|(u_1-\varphi_1, u_2-\varphi_2)\|_{\Y},
\end{array}
$$
because of the choice of $R$ as in \eqref{eqfixedconstant}, and hence $\Lambda$ is a contraction map from $\Y_R$ into itself, for the choice of $R$ as in \eqref{eqfixedconstant}. 
Now, using the Banach fixed point theorem, we obtain $(w, \psi)\in \Y_R$, the fixed point of $\Lambda$, and therefore, $(w, \psi)\in \Y_R$ satisfies \eqref{non lin} with a control $h\in \V$ and with initial data $(w_0, \psi_0)\in (L^2(0, 1))^2$ satisfying \eqref{eqiniest}, for $\mu$ as determined in \eqref{eqfixedconstant}. 

Since $(w, \psi)\in \Y_R$ and $\displaystyle \lim_{t\to T^{-}} \rho_0(t)=0$, it yields
	$$ \left(w(T,\cdot), \psi(T,\cdot)  \right) =(0,0) \quad \mathrm{in}\quad L^2(0, 1).$$
This proves the local null-controllability of the nonlinear system \eqref{non lin}.

\section{Further remarks and open problems}\label{Sec: remarks}
In view of our results, a few questions arise naturally and we mention them here for future investigations. 

\emph{Non-constant steady state:} In this article, we study the controllability of \eqref{CHB} around a certain steady state $(\bar{u}, \bar{\phi})$, where $\bar{\phi}$ is a constant and $(\bar{u}, \bar{\phi})$ satisfies the conditions mentioned in \Cref{rem-assump-steadystate}. It would be more interesting to consider the steady state satisfying \eqref{steady CHB} and both $ \bar{u}$ and $\bar{\phi}$ are non-constant functions of space variable. Then the system around $(\bar{u}, \bar{\phi})$ is in the form: 
\begin{equation}\label{gen non lin}
		\begin{cases}
			w_t-\gamma w_{xx}+\bar u w_x+\bar u' w=\gamma_1 \psi_x+ (12\bar{\phi}^2\bar{\phi}'-\bar{\phi}'-\bar{\phi}'')\psi- \bar{\phi}'\psi_{xx}
			+\widetilde{N}_1(w,\psi) +\chi_{\mathcal{O}_1} h_1, & \text{in } Q_T, \\
			\psi_t+ \psi_{xxxx}+\gamma_2 \psi_{xx} + (\bar{u}-24\bar{\phi}\bar{\phi}')\psi_x-(24\bar{\phi}'^2+2\bar{\phi}\bar{\phi}'')\psi
			=-\phi'w+ \widetilde{N}_2(w,\psi)+\chi_{\mathcal{O}_2} h_2,  & \text{in } Q_T,\\
			w(t, 0)=0,  \  w(t, 1)=0,  & t \in (0, T), \\
			\psi_x(t,0)=0, \  \psi_x(t,1)=0, \quad \psi_{xxx}(t, 0)=0,  \  \psi_{xxx}(t, 1)=0 & t \in (0, T),    \\
			w(0, x)=w_0(x),\, \psi(0,x)=\psi_0(x), & x \in  (0, 1),
		\end{cases}
	\end{equation} 
where $\widetilde{N}_1$ and $\widetilde{N}_2$ are nonlinear terms. Here $h_1$ and $h_2$ are controls acting in any open subsets $\mathcal{O}_1$ and $\mathcal{O}_2$ of $(0, 1)$ respectively. Comparing \eqref{non lin}, it can be observed that in $\eqref{gen non lin}_1$, a term involving $\psi_{xx}$ is present whereas a term involving $w$ is present in 
$\eqref{gen non lin}_2$. Thus, the coupling between the equations are stronger than that of \eqref{non lin} and this makes the problem more challenging. The available Carleman inequality may not be easily adaptable to this case because of this strong coupling. Moreover, if the controllability result can be obtained, it will be interesting to investigate if the controllability can be obtained using a minimal number of controls. 

It is interesting to study whether one can establish a local controllability result when the control acts in the second-order parabolic equation. It is worth mentioning that the work \cite{CE16} focused on this control issue for the stabilized KS equation by controlling through the heat equation. The main ingredient of their work is to explore a new Carleman inequality associated to the KS (fourth order) equation with nonhomogeneous boundary data and source terms. The first step for this problem related to the Cahn-Hilliard-Burgers' equation will be to investigate whether this type of Carleman inequality holds for the Cahn-Hilliard equation.

\smallskip

\emph{Multi-dimensional case:} The immediate interest would be to study the controllability of the Cahn-Hilliard-Navier-Stokes equation \eqref{gen} in higher dimension. But it seems that a major obstacle to use our method to obtain the controllability of CHNS is the Carleman estimate for $\eqref{gen}_1-\eqref{gen}_2$ with the boundary conditions $\eqref{gen}_6$. The Carleman inequality for this case is very much related to that of the Cahn-Hilliard equation as mentioned in \cite[Section 5.2]{Guzman}. As indicated in \cite[Section 5.2]{Guzman}, the required weight function to obtain the Carleman inequality for the Cahn-Hilliard equation with this certain boundary condition is not available and hence obtaining the inequality is still an open problem. Thus, to study the controllability of the CHNS model in higher dimension will need a different approach or an adaptation of the Carleman inequality in this set-up. 

\smallskip

 \emph{Considering different potential $F$:}   In this article we have considered the double well potential  of regular type; indeed a particular example of a regular potential given in \eqref{reg potential}. For diffuse interface systems modeled by these equations, a potential of singular type is more physically and thermodynamically relevant.  We note that CHNS system for well-posedness as well as optimal control are well studied in the literature for the following singular potential called as  logarithmic potential
 \begin{align}
      F_{log}(s) =\frac{\theta}{2} \big((1+s)\log(1+s)+(1-s)\log(1-s)\big) + \frac{\theta_c}{2}(1-s^2) \quad s\in (-1,1), \quad 0 < \theta < \theta_c,
\end{align}
where $\theta$ and $\theta_c$ are proportional to absolute and critical temperature, respectively. Obvious extension of the current work is to consider logarithmic potential. Note that $F_{log}$ has singularities at $-1, +1$.  Hence deriving corresponding linearized system similar to \eqref{gen non lin}
system itself will be challenging. One will have to follow a  general treatment for such problems  via approximating the singular potential by sequence of regular potentials. However for controllability results the convergence  arguments may need careful treatment.


\smallskip 
\emph{Boundary control problems: } Another interesting aspect can be addressing the boundary null controllability problem for the Cahn-Hilliard-Burgers' equation. In \cite{CE12}, a local null controllability result for a system coupling KS-KdV and heat equations has been established using three boundary controls acting at the left end of the spatial domain $(0,1)$. This naturally raises the question of whether a similar controllability result could be achieved for the system considered in this paper. A challenging problem is whether the controllability result can be improved using less number of controls on the boundary. It is well-known that the usual Carleman estimate is not effective to explore the boundary controllability problem for coupled system with less number of control than variables as we do not have the interaction of the boundary terms for the associated adjoint system to absorb the unusual observation terms from the joint Carleman estimate. Thus, this question offers promising avenues for further research.


\section{Acknowledgment}
{The authors would like to express their sincere thanks to the reviewers for their helpful comments, valuable questions, and suggestions, which significantly improved the first version of the paper. Subrata Majumdar gratefully acknowledges Víctor Hernández-Santamaría for pointing out the reference \cite{TTW23} and for the fruitful discussions.}

\appendix
\section{Some auxiliary result}\label{rev_app}
{In this section we prove the following result:
\begin{proposition}
	Let $\gamma>0,$ $T>0$. Assume $f\in L^2(0,T;H^1_0(0,1)),$ and $\bar u\in H^2(0,1)\cap H^1_0(0,1).$ Then the equation
	\begin{equation}\label{rev_heat}
		\begin{cases}
			z_t-\gamma z_{xx}-\bar u z_x  =  f &\text{ in }   (0,T) \times (0,1),\\
			z(t,0)=z(t,1)=0 & \text{ in }   (0,T),\\
			z(0,x)=0 & \text{ in }   (0,1).
		\end{cases}
	\end{equation}
has a unique solution $z\in L^2(0,T;H^3(0,1))$ and it satisfies the following estimate
\begin{align}\label{rev_h3}
	\norm{z}_{L^2(0,T;H^3(0,1))}\leq Ce^{CT}\norm{f}_{L^2(0,T;H^1_0(0,1))},
\end{align}
where $C>0$ independent of $T.$
\end{proposition}
\begin{proof}
For large enough $\lambda>0$ and $\bar{u}\in H^2(0, 1)\cap H^1_0(0, 1)$, let us consider the equation
\begin{equation}\label{rev_heat1}
		\begin{cases}
			\widetilde{z}_t-\gamma \widetilde{z}_{xx}-\bar u \widetilde{z}_x +\lambda \widetilde{z} = \widetilde{f} &\text{ in }   (0,T) \times (0,1),\\
			\widetilde{z}(t,0)=\widetilde{z}(t,1)=0 & \text{ in }   (0,T),\\
			\widetilde{z}(0,x)=0 & \text{ in }   (0,1),
		\end{cases}
	\end{equation}
where $$\widetilde{f}(t, x)= \left\{\begin{array}{ll} e^{-\lambda t} f(t, x) & \forall\, (t, x) \in (0, T)\times (0, 1) \\
 0 & \forall\, (t, x)\in (T, \infty)\times (0, 1),\end{array}\right.$$ and $f$ is any given function in $L^2(0, T; H^1_0(0, 1))$. 

The linear operator associated  to \eqref{rev_heat1} in $H^1_0(0, 1)$ is given by 
$$ 
\begin{array}{l}
\mathcal{D}(\mathcal{A}_1)=\{\widetilde{z}\in H^3(0, 1)\mid \widetilde{z}(0)=0=\widetilde{z}(1),  \quad \widetilde{z}''(0)=0=\widetilde{z}''(1)\}, \\
\mathcal{A}_1 \widetilde{z}= \gamma \widetilde{z}''+\bar{u} \widetilde{z}'-\lambda \widetilde{z}, \quad \widetilde{z}\in \mathcal{D}(\mathcal{A}_1). 
\end{array}
$$ 
It can be shown that $(\mathcal{A}_1, \mathcal{D}(\mathcal{A}_1))$ forms a stable analytic semigroup on $H^1_0(0, 1)$ and hence from \cite[Proposition 3.7]{MR2273323}, it follows that \eqref{rev_heat1} admits a unique solution $\widetilde{z}\in L^2(0, \infty; H^3(0, 1))\cap H^1(0, \infty; H^1_0(0, 1))$ satisfying 
\begin{equation}\label{eqrevh2} \|\widetilde{z}\|_{L^2(0, \infty; H^3(0, 1))}\le C \|\widetilde{f}\|_{L^2(0, \infty; H^1_0(0,1))},
\end{equation}
for some positive constant $C$. 

Now introducing $z(t, x)=e^{\lambda t}\widetilde{z}(t, x)$, we obtain that $z\in L^2(0, T; H^3(0, 1))\cap H^1(0, T; H^1_0(0, 1))$ is the unique solution of \eqref{rev_heat} satisfying 
$$ \|z\|_{L^2(0, T; H^3(0, 1))}\le C e^{\lambda T} \|f\|_{L^2(0, T; H^1_0(0,1))},$$
where $C$ is same as the constant obtained in \eqref{eqrevh2} and is independent of $T$. Hence the result follows. 
\end{proof}}

\bibliographystyle{plain} 
\bibliography{ref1}
\end{document}